\documentclass[11pt,tbtags]{amsart}

\usepackage{amsmath, amsthm, amssymb, amsfonts, enumerate}
\usepackage{graphicx}
\usepackage{wasysym, pdfpages}
\usepackage{epstopdf}
\usepackage[colorlinks=true,linkcolor=blue,urlcolor=blue]{hyperref}
\usepackage{dsfont}
\usepackage{todonotes}

\usepackage{geometry}
\newcommand{\RR}{{\mathbb{R}}}

\newcommand{\EE}{{\mathbb{E}}}
\newcommand{\NN}{{\mathbb{N}}}
\newcommand{\PP}{{\mathbb{P}}}
\newcommand{\WW}{{\mathbb{W}}}

\newcommand{\CF}{{\mathcal{F}}}
\newcommand{\CG}{{\mathcal{G}}}

\newcommand{\CR}{{\mathcal{R}}}

\newcommand{\CA}{{\mathcal{A}}}
\newcommand{\CC}{{\mathcal{C}}}
\newcommand{\CI}{{\mathcal{I}}}

\newcommand{\CS}{{\mathcal{S}}}
\newcommand{\CL}{{\mathcal{L}}}
\newcommand{\CT}{{\mathcal{T}}}

\newcommand{\indic}{\ensuremath{\mathds 1}}
\newcommand{\comment}[1]{}

\newtheorem{thm}{Theorem}[section]
\newtheorem{cor}[thm]{Corollary}
\newtheorem{lem}[thm]{Lemma}
\newtheorem{pro}[thm]{Proposition}

\newtheorem{rem}[thm]{Remark}

\newtheorem{rems}[thm]{Remarks}

\newtheorem{ass}[thm]{Assumption}
\def \Proof{\noindent{\bf Proof:} \ }

\def\eps{\varepsilon}

\def \exp{ {\rm exp} }

\def \lim{ {\rm lim} }

\let\inf\relax \DeclareMathOperator*\inf{\vphantom{p}inf}
\def \tq{\, |\,}

\def \egal {\stackrel{def}{=}}

\title[A Dynkin game with incomplete information]{A Dynkin game on assets with incomplete information on the return}
\author[Tiziano De Angelis, Fabien Gensbittel, St\'ephane Villeneuve]{Tiziano De Angelis, Fabien Gensbittel, St\'ephane Villeneuve}
\address{T.~De Angelis: School of Mathematics, University of Leeds, Woodhouse Lane, LS2 9JT Leeds, UK.}
\address{F.~Gensbittel and S.~Villeneuve: Toulouse School of Economics (TSE-R, Universit\' e Toulouse 1 Capitole), 21 all\' ee de Brienne, 31000
Toulouse, France.}
\email{\href{mailto:t.deangelis@leeds.ac.uk}{t.deangelis@leeds.ac.uk}}
\email{\href{mailto:fabien.gensbittel@tse-fr.eu}{fabien.gensbittel@tse-fr.eu}} 
\email{\href{stephane.villeneuve@tse-fr.eu}{stephane.villeneuve@tse-fr.eu}} 
\date{\today}
\thanks{{\em Acknowledgments}: T.~De Angelis was partially supported by the EPSRC grant EP/R021201/1. \\
We thank an anonymous referee whose insightful comments contributed to the discussion in Section \ref{sec:concl}.}

\numberwithin{equation}{section}

\begin{document}     
  
\begin{abstract} 
This paper studies a 2-players zero-sum Dynkin game arising from pricing an option on an asset whose rate of return is unknown to both players. Using filtering techniques we first reduce the problem to a zero-sum Dynkin game on a bi-dimensional diffusion $(X,Y)$. Then we characterize the existence of a Nash equilibrium in pure strategies in which each player stops at the hitting time of $(X,Y)$ to a set with moving boundary. A detailed description of the stopping sets for the two players is provided along with global $C^1$ regularity of the value function. 
\end{abstract}

\maketitle

\section{Introduction}

Zero-sum optimal stopping games (Dynkin games) have received a lot of attention since the seminal paper by Dynkin \cite{dynkin}, see also the classical references \cite{bensoussanfriedman} and \cite{lepeltiermaingueneau}.  In particular, these games have found applications in mathematical finance where the arbitrage-free pricing of American options  with early cancellation (game options) relies on the computation of the value of a zero-sum game of optimal stopping between the buyer and the seller (see \cite{kifer},\cite{kyprianou}). A common assumption in the financial application of Dynkin games is that the players have complete information about the parameters of the underlying stochastic process. In practice, however, there are many situations in which parameters are difficult to estimate and in particular this is true for the drift of the process. 

Our work is inspired by the real option literature, where the value of an investment (like the beginning of the extraction of a natural resource or the investment in a R\&D programme) is a contingent asset, depending on the price $S$ of some underlying asset, and it is computed by using arbitrage arguments (see \cite{DixitPindyck}). It is known that the problem itself  boils down to an optimal timing decision, hence optimal stopping is the key mathematical tool. Following \cite{DixitPindyck}, we assume that the price process evolves according to a geometric Brownian motion
\[
\frac{dS_t}{S_t}=\mu\,dt+\sigma dB_t
\]
where $\mu$ is the log-return on the so-called risk-adjusted asset price.

The capital asset pricing model allows us to determine the risk-adjusted discount rate $r$ which is used to discount future cashflows (notice that this is in general larger than the risk-free rate, see, e.g., \cite[p.~178]{DixitPindyck}). In line with \cite{DixitPindyck} we assume $\mu\le r$ and denote the difference $r-\mu$ by $\delta_0$. The condition $\mu\le r$ avoids that the value of an investment project whose payoff is linear in $S$ becomes unbounded (which would lead the investor to delay the investment forever). It is known that estimating the return of the risk-adjusted price of an asset is a challenging task and we embed this feature in our model by considering an asset with a partially unobservable drift $\mu$.

A typical problem that we have in mind is the one of a firm holding a concession to drill oil wells. Being aware of the social costs and benefits of the oil field development, a public authority would like to sign a contract where the concession rights can be cancelled at any time, pending the payment of a contractual penalty. From the investor's point of view (and simplifying the model for the benefit of tractability) the decision to invest would be profitable only if the value of the underlying commodity can compensate for the fixed cost of investment $K>0$. In this sense one can interpret the option to invest as a Call option on the price of the commodity, with strike equal to $K$. A cancellation of the agreement would require a payment equal to the Call payoff plus a penalty (i.e., to compensate for the lost investment opportunity).  

Motivated by the above considerations, in this paper we study zero-sum optimal stopping games with incomplete information about the return of the underlying asset. We are interested in the existence of the value  as well the existence and characterization of Nash equilibria for the game. To enable a detailed theoretical analysis, we shall keep the real option model simple while, at the same time, drawing from the vast literature on Israeli options (initiated by \cite{kifer}).

We assume that the buyer (player 1) and the seller (player 2) of a Call option on an asset $S$ agree on a constant risk-adjusted discount rate $r>0$, which is used to discount future payoffs in the game (i.e., we assume that players have the same belief on the future of the economy). Moreover we model the uncertainty on the the asset return by assuming that the adjusted log-return is random and only partially observable. To avoid confusion with the previously introduced notation we denote it by $\tilde\mu$ (as opposed to $\mu$ in the previous page). In particular we assume $\tilde \mu=r-\delta_0 D$, where $\delta_0>0$ is a constant and $D\in\{0,1\}$ is random and unobservable to the players. 

Our choice for $\tilde \mu$ ties up nicely with the usual concept of net return on a stock paying dividends at a rate $\delta_0 D$. Although other choices for $\tilde\mu$ are clearly possible, we shall see below that this basic model already poses significant mathematical challenges. To the best of our knowledge this is the first paper addressing a zero-sum game with partial information via a probabilistic analysis of the related free boundary problem, hence we leave other parameter choices for future work.

The asset in our model evolves, on a probability space $(\Omega,\CF,\PP)$, according to 
\begin{align}\label{BS}
dS_t= (r-\delta_0 D)S_tdt+\sigma S_t dB_t, \quad S_0=x>0,  
\end{align}
where $(B_t)_{t\ge 0}$ is a Brownian motion and $\sigma>0$ is the volatility. The random variable $D$ takes the values $0$ or $1$ with $\PP(D=1)=y$ and it is assumed to be independent of $(B_t)_{t\ge 0}$. We denote by $\CF^{S}:=(\CF^{S})_{t\ge 0}$ the filtration generated by the observed process $S$ and by $\overline{\CF}^S:= (\overline{\CF}^{S})_{t\ge 0}$ its augmentation with $\PP$-null sets (see further details in Section \ref{sec:dyn}). Then we define by $\CT^S$ the set of $\overline{\CF}^{S}$-stopping times. 

In our game we fix $K>0$ and $\eps_0>0$ and let
\begin{align}\label{G1G2}
G_1(x):=(x-K)^+,\qquad G_2(x):=(x-K)^++\eps_0
\end{align}
be the payoff for player 1 (the option holder) and the cost of cancellation for player 2 (the seller), respectively. Then the formulation of our game is the following: the expected discounted payoff of the game is
\begin{equation} \label{gameformulationone}
 M_{x,y}(\tau,\gamma) =\EE[ e^{-r\tau}G_1(S_{\tau})\indic_{\{\tau \le \gamma\}}+ e^{-r \gamma}G_2(S_{\gamma})\indic_{\{\gamma < \tau\}} ]
\end{equation}
where $\tau,\gamma\in\CT^S$. In particular the option holder picks $\tau$, in order to exercise the option, and the seller picks $\gamma$, in order to cancel it. The holder aims at maximising her revenue while the seller wants to minimise costs.  By convention, we set 
\[
e^{-r\tau}G_1(S_{\tau})\indic_{\{\tau = \infty\}}=e^{-r \gamma}G_2(S_{\gamma})\indic_{\{\gamma =\infty\}}=0,\quad\PP-a.s.
\]

The notation $M_{x,y}$ accounts for the dependence of the stopping functional on the initial asset value and on the a-priori probability of the event $\{D=1\}$. This notation will be fully justified and explained in Section \ref{sec:dyn} below.

As usual we define the upper value and the lower value of the stopping game, respectively by
\begin{align}\label{upVloV}
\overline{V}(x,y)= \inf_{\gamma }\; \sup_{\tau }\; M_{x,y}(\tau,\gamma)\quad \hbox{ and } \quad \underline{V}(x,y)=  \sup_{\tau }\; \inf_{\gamma }\; M_{x,y}(\tau,\gamma).
\end{align}
When $\underline{V}(x,y)=\overline{V}(x,y)$, the game has a value $V(x,y):=\underline{V}(x,y)=\overline{V}(x,y)$. Moreover, if there exist two stopping times $(\tau_*,\gamma_*)$ such that
\begin{align*}
M_{x,y}(\tau,\gamma_*) \le M_{x,y}(\tau_*,\gamma_*) \le M_{x,y}(\tau_*,\gamma)
\end{align*}
for all stopping times $\tau$ and $\gamma$, the pair $(\tau_*,\gamma_*)$ is a saddle point or a {\it Nash equilibrium} for the optimal stopping game and in that case the game has a value with $V(x,y)=M_{x,y}(\tau_*,\gamma_*)$.

In the context of Israeli options one has $\PP(D=1)=1$ or $\PP(D=1)=0$ (the non-dividend case).
Explicit computations have been established by \cite{ekstromvilleneuve} and \cite{yam} in the perpetual case. Both papers show that the dividend parameter $\delta_0$ plays an important role for the existence of an equilibrium in the game and this will be the case also in the present work.

We recall now some results from the existing literature so that we can later discuss the mathematical novelty of our work.
The \emph{existence of the value} for optimal stopping games with multi-dimensional Markov processes was proved in \cite{ekstrompeskir} using martingale methods and by Bensoussan and Friedman \cite{bensoussanfriedman} via variational inequalities. These methods require suitable integrability of the payoff processes, i.e., in our notation, the processes $e^{-rt}G_i(S_t)$, $i=1,2$ must be uniformly integrable. When such condition is not fulfilled, the existence of the value was proven in \cite{ekstromvilleneuve} but only for one-dimensional diffusions. Results in \cite{ekstromvilleneuve} rely upon a generalized type of concavity introduced in \cite{dynkinyushkevich} and brought up to date in \cite{dayanikkaratzas}.

On the other hand sufficient conditions for the \emph{existence of Nash equilibria} in Markovian setting have been studied in \cite{ekstrompeskir} and \cite{ekstromvilleneuve}. For a rather general class of Markov processes these conditions include the above mentioned uniform integrability of the payoff processes. In the special case of one-dimensional diffusions weaker integrability may instead be sufficient (see \cite{ekstromvilleneuve}, Proposition 4.3).

In our setting we are faced with two main technical difficulties in establishing existence of the value and of a Nash equilibrium: (i) the process $S$ is not Markovian and (ii) it fails to fulfil the condition of uniform integrability (see Remark \ref{notUI}), in particular for any initial condition $S_0=x\in\RR_+$ we have
\begin{equation}\label{uniformbound}
\EE\left(\sup_{0\le t < \infty} e^{-rt}G_i(S_t)\right)=+\infty,\quad i=1,2.
\end{equation}

To overcome the first difficulty we rely upon filtering theory and increase the dimension of our state space. Informally we could say that we take into account the progressive update of the players' estimate on $D$, based on the observation of $S$. This approach leads us to study a two dimensional Markovian system which we denote by $(X_t,Y_t)_{t\ge 0}$, where (at least formally) $X=S$ and $Y_t=\EE[D|\CF^S_t]$. On the other hand, to tackle the lack of uniform integrability and prove the existence of the value of the game, we adapt methods developed by Lepeltier-Maingueneau \cite{lepeltiermaingueneau} and Ekstrom-Peskir \cite{ekstrompeskir}. 

After we prove existence of the value, we are then in the position to carry out a detailed analysis of the structure of the stopping sets for the two players, i.e.~the sub-sets of the state space in which $V=G_i$, $i=1,2$. 
Denoting $\CS_i:=\{V=G_i\}$, $i=1,2$ we study properties of the boundaries of $\CS_1$ and $\CS_2$ which we subsequently use to state conditions for the existence of a saddle point (Nash equilibrium). The latter is provided in terms of hitting times to $\CS_1$ and $\CS_2$. 

In our analysis we use two equivalent representations of the two-dimensional dynamics. These are linked to one another by a deterministic transformation -- the so-called reduction of second order PDEs to normal form (observe that a similar transformation was already used by several papers like \cite{DMV},\cite{ekstromlu} and \cite{gapeevshiryaev} among others). Indeed we first observe that the process $(X_t,Y_t)$ is driven by only one Brownian motion and it is therefore degenerate; then we perform a change of coordinates to obtain a new process $(Z_t,Y_t)$. Here $Z_t$ is deterministic and either increasing or decreasing, depending on the choice of parameters in the problem. Effectively the process $Z$ plays the role of a `time' process. 

We would like to emphasize that the probabilistic study of free boundary problems related to zero-sum Dynkin games on two dimensional diffusions has not received much attention so far. Works in this direction but in a parabolic setting are \cite{DeAFe14} and \cite{yam}. Our analysis here goes beyond results in those papers by showing for example that the value of the game is a globally $C^1$ function of the state variables $(x,y)$. This type of regularity is much stronger than the well-known \emph{smooth-fit}, which gives continuity of one directional derivative with respect to one state variable. Related work on $C^1$ regularity is contained in \cite{DeAPe18}, which however does not cover our game setting.

The outline of the paper is as follows. In Section 2, we specify the model and provide a Markovian formulation of the zero-sum game. Existence and continuity of the value for the game \eqref{gameformulationone} is obtained in Section 3. The geometry of the stopping sets is obtained in Section 4, in the $(x,y)$-plane, and in Section 5, in the $(z,y)$-plane (parabolic formulation). Hitting times to those sets are used in Section 6 to prove higher regularity of the value, e.g.~its global $C^1$ regularity,  and in Section 7 we obtain sufficient conditions for the existence of a saddle point. Finally, in Section 8 we collect some concluding remarks and discuss possible directions for further research. A few technical results are given in Appendix.

\section{Dynamics of the underlying asset}\label{sec:dyn}

We begin by considering the probability space $\Omega^0=C([0 +\infty), \RR) \times \{0,1\}$ endowed with a sigma algebra $\CF^0$ and the product probability $\PP^0_y=\WW \otimes \pi(y)$ where $\WW$ is the standard Wiener measure and $\pi(y)=(1-y,y)$. Let us denote $((B_t)_{t\geq0},D)$ a canonical element of $\Omega^0$ and let $r>0$, $\delta_0 > 0$ and $\sigma >0$ be fixed. Then the asset's value (with uncertain return rate) which is described by \eqref{BS} has an explicit expression in terms of the couple $(B,D)$, i.e. 
\begin{equation}\label{defiX}
S^x_t = x e^{\sigma B_t + (r-\delta_0 D)t - \frac{t \sigma^2}{2}}
\end{equation}
The process $S^x$ is a geometric Brownian motion whose drift parameter depends on the unobservable random variable $D$. We recall that the latter is independent of the Brownian motion $B$. As discussed in the introduction a technical difficulty arising in our model is the lack of uniform integrability of the process $S^x$.
\begin{rem}\label{notUI}
If $y \in(0,1)$, the process $e^{-rt}S^x_t$ is not uniformly integrable because
\begin{align*}
\lim_{t\to+\infty}e^{-rt}S^x_t = \lim_{t\to+\infty}x e^{\sigma B_t - \frac{t \sigma^2}{2}}\left(\indic_{\{D=0\}} + \indic_{\{D=1\}} e^{-\delta_0 t}\right)=0, \:\:\:\PP^0_y-\text{a.s.,} 
\end{align*} 
whereas 
\begin{align*}
\lim_{t\to+\infty}\EE^0_y[ e^{-rt}S_t^x]=(1-y)x.
\end{align*}
Hence, by linearity of the payoffs $G_i$, $i=1,2$ in \eqref{G1G2} we obtain \eqref{uniformbound}.
\end{rem}

We aim at giving a rigorous formulation for the game call option \eqref{gameformulationone}. One way to do it is to replace $\PP$ and $(S_t)_{t\ge 0}$ in \eqref{gameformulationone} by $\PP^0_y$ and $(S^x_t)_{t\ge 0}$ defined above. Let $\CF^{S}:=(\CF^{S})_{t\ge 0}$ be the filtration generated by $S^x$ and $\overline{\CF}^{S}:=(\overline{\CF}^{S})_{t\ge 0}$ be its augmentation with $\PP^0_y$-null sets. Then we denote $\CT^S$ the set of $\overline{\CF}^{S}$-stopping times and the optimisation is taken over stopping times $(\tau,\gamma)\in\CT^S$. 
The disadvantage of this formulation is that the dynamics of $S^x$ is not Markovian and therefore for the solution of the problem we cannot rely upon free boundary methods. To overcome this difficulty we want to reduce our problem to a Markovian framework by using filtering techniques.\\ 

According to \cite[Thm.~2.35, p.~40]{BainCrisan} the filtration $\overline{\CF}^{S}:=(\overline{\CF}^{S})_{t\ge 0}$  is right continuous and therefore satisfies the usual assumptions. Thus, we define the process $(D^y_t)_{t\geq 0}$ as an $\CF^{S}$-c\`{a}dl\`{a}g version of the martingale $(\EE^0_y[D  | \overline{\CF}^{S}_{t}])_{t\geq 0}$. Notice that $(D^y_t)_{t \geq 0}$ is a bounded martingale that converges almost surely to $D$. The latter is $\overline{\CF}^S_{\infty}$-measurable because 
\[ 
\frac{1}{t}\ln(S_t^x)=\frac{1}{t}\left[\ln(x)+ (\sigma B_t - \frac{t \sigma^2}{2})-\delta_0 t\indic_{\{D=1\}}\right] \underset{t \rightarrow \infty}{\longrightarrow} \frac{\sigma^2}{2}-\delta_0\indic_{\{D=1\}}. 
\]

According to Chapter 9 in Liptser-Shiryaev \cite{LipsterShiryaev} (see also Chapter 4.2 in Shiryaev \cite{Shiryaev}), the process $(S^x,D^y)$ is the unique strong solution to the following SDE,
\begin{align}\label{SDE1} 
\left\{ 
\begin{array}{l} 
dS^x_t  =  (r-\delta_0 D^y_t) S^x_t dt + \sigma S^x_t d\widehat{W}_t \\ [+4pt]
dD^y_t  =  -\frac{\delta_0}{\sigma} D^y_t (1-D^y_t) d \widehat W_t 
\end{array} 
\right.
\end{align}
where 
\begin{align*}
\widehat {W}_t=\frac{1}{\sigma}\left(\int_0^t (S^x_u)^{-1} dS^x_u- \int_0^t (r-\delta_0 D^y_u)du\right)
\end{align*}
is an $\overline{\CF}^{S}$-adapted Brownian motion under $\PP^0_y$. The couple $(S^x,D^y)$ is therefore adapted to the augmentation of the filtration generated by $\widehat W$, which we denote by $\overline{\CF}^{\widehat W}$. This implies in particular $\overline{\CF}^{S}\subseteq \overline{\CF}^{\widehat W}$ and $\overline{\CF}^{\widehat W}=\overline{\CF}^S$ because $\widehat W$ is $\overline{\CF}^{S}$-adapted. Notice also that the process $(D^y_t)_{t \geq 0}$ is adapted to the filtration $\overline{\CF}^S$ by construction, so that it is no surprise that the new Brownian motion $\widehat W$ is also adapted to $\overline{\CF}^{S}$.

Above we have obtained $(S^x,D^y)$ on the space $(\Omega^0,\CF^0,\PP^0_y)$ which depends on the probability distribution of the random variable $D$. We prefer to get rid of such dependence and consider another process $(X_t,Y_t)_{t\ge0}$, having the same law than $(S^x_t,D^y_t)_{t\ge0}$, but defined below on a new probability space.

Take a probability space $(\Omega,\CF,\PP)$, denote by $W:=(W_t)_{t\ge0}$ a Brownian motion on this space and by $\mathbb{F}:=(\CF_t)_{t\ge 0}$ the augmentation of the filtration that it generates. For $(x,y)\in\RR_+\times(0,1)$, let $(X,Y)$ be the unique strong solution of the bi-dimensional SDE
\begin{align}\label{XY}
\left\{ 
\begin{array}{ll} 
dX_t =  (r- \delta_0 Y_t) X_t dt + \sigma X_t dW_t, & X_0=x,  \\ [+4pt]
dY_t = -\frac{\delta_0}{\sigma} Y_t (1-Y_t) dW_t, & Y_0=y. 
\end{array} 
\right. 
\end{align}
To keep track of the initial point we use the notation $(X^{x,y},Y^y)$ and notice that by standard theory $(t,x,y)\mapsto(X^{x,y}_t,Y^y_t)$ is indeed continuous $\PP$-almost surely. Notice also that the second equation is independent of the first one and therefore its solution, $Y^y$, is independent of $x$.

Since the processes 
$\{(\Omega,\CF, \mathbb{F},\PP),(X^{x,y},Y^y)\}$ and $\{(\Omega^0,\CF^0, \overline \CF^{\widehat W},\PP^0_y),(S^x,D^y)\}$
have the same law then the game option is more conveniently formulated using the former since it is Markovian and the probability measure is independent of $y$. This will be done in the next section.

Often in what follows we use the notation $\PP_{x,y}(\,\cdot\,)=\PP(\,\cdot\,|X_0=x,Y_0=y)$ and drop the apex in the couple $(X,Y)$.
Before closing the section we notice that for all $t\ge 0$
\begin{align}\label{def:Xhat}
X^{x,y}_t=x \, \exp\left( \int_0^t (r-\delta_0 Y_s^y-\frac{\sigma^2}{2})\,ds+\sigma W_t\right),\quad\PP-\text{a.s.}
\end{align}
Moreover we recall that since $(\zeta_t)_{t\ge 0}:=(e^{-r t}X_t)_{t\ge0}$ is a continuous super-martingale, with last element $\zeta_\infty:=\lim_{t\to\infty}\zeta_t=0$, the optional sampling theorem guarantees (see \cite[Thm.1.3.22]{KS})
\begin{align}\label{op-sam}
\EE_x\left[e^{-r\rho}X_\rho\vert\CF_\nu\right]\le e^{-r\nu}X_\nu, \quad\PP_x-\text{a.s.}
\end{align}
for all stopping times $\rho\ge\nu$.

\section{The game and its value}

The payoffs $G_i$, $i=1,2$ in \eqref{G1G2} are non-decreasing and $1$-Lipschitz continuous on $\RR_+$ with $0 \leq G_1 < G_2$. It is also clear that 
\begin{align}\label{infty}
\lim_{t \rightarrow \infty} e^{-rt}G_i(X_t^{x,y}) =0,\qquad \PP-a.s. 
\end{align}
for any $(x,y)\in\RR_+\times(0,1)$, due to the first formula in Remark \ref{notUI}. We now recall the formulation of the game expected payoff \eqref{gameformulationone} given in the Introduction and notice that, thanks to the equivalence explained in the previous section, we can rewrite it as 

\begin{equation}\label{gameformulationtwo}
  M_{x,y}(\tau,\gamma) =\EE\left[ e^{-r\tau}G_1(X^{x,y}_{\tau})\indic_{\{\tau \leq \gamma\}}+ e^{-r \gamma}G_2(X^{x,y}_{\gamma})\indic_{\{\gamma < \tau\}} \right].
 \end{equation}
The stopping times $(\tau,\gamma)$ are drawn from the set $\CT$ of $\mathbb F$-stopping times and the dependence of $M_{x,y}(\tau,\gamma)$ on $(x,y)$ is clearly expressed. Thanks to \eqref{infty} on the event $\{\tau\wedge\gamma=+\infty\}$ we simply get a zero payoff for both players.

We recall here that player 1 (the buyer) picks $\tau$ in order to maximise \eqref{gameformulationtwo}, whereas player 2 (the seller) chooses $\gamma$ in order to minimise \eqref{gameformulationtwo}. 
The upper value $\overline V$ and the lower value $\underline V$ of the game are expressed as in \eqref{upVloV}. We spend the rest of this section proving that these functions indeed coincide so that the game has a value $V$.
\vspace{+3pt}

We start by proving some regularity result of $\overline V$ and $\underline V$.
\begin{lem}\label{Lipschitz}  The functions $\overline V$ and $\underline V$ are: 
\begin{itemize}
\item[  (i)] non-decreasing with respect to (w.r.t.)~$x$ and non-increasing w.r.t.~$y$
\item[ (ii)] $1$-Lipschitz w.r.t.~$x$, uniformly w.r.t.~$y\in[0,1]$
\item[(iii)] locally Lipschitz w. r. t. $y$, i.e.~for $f=\underline{V}$ or $f=\overline{V}$ and a given constant $C>0$ we have
\[|f(x,y)-f(x,y')| \leq C(1+|x|) |y-y'|, \quad \forall x>0, \forall y,y'\in [0,1]. \]
\end{itemize}
\end{lem}  

\begin{proof}[Proof of Lemma \ref{Lipschitz}]
Without loss of generality, we only provide full details for $\underline{V}$.
\vspace{+3pt}

[\emph{Proof of (i)}]
Let us first prove monotonicity with respect to $x$.  Fix $y\in(0,1)$ and $x \geq x'$, then for any $\varepsilon>0$, there exist a couple $(\tau_\varepsilon,\gamma_\varepsilon)$ such that 
\begin{align}
M_{x,y}(\tau_\eps,\gamma_\eps)\le\underline V(x,y)+\frac{\eps}{2}\quad\text{and}\quad M_{x',y}(\tau_\eps,\gamma_\eps)\ge \underline V(x',y)-\frac{\eps}{2}. 
\end{align}
Therefore we also have
\begin{align*} 
\underline{V}(x,y)- \underline{V}(x',y) \geq& M_{x,y}(\tau_\varepsilon,\gamma_\varepsilon) - M_{x',y}(\tau_\varepsilon,\gamma_\varepsilon) - \varepsilon\\
=&\EE\left[ e^{-r(\tau_\varepsilon\wedge\gamma_\eps)}((X^{x,y}_{\tau_\varepsilon\wedge\gamma_\eps}-K)^+- (X^{x',y}_{\tau_\varepsilon\wedge\gamma_\eps}-K)^+)\right]-\eps\\
\ge & -\eps
\end{align*}
where the last inequality follows by observing that $X^{x,y}_t \ge X^{x',y}_t$, $\PP$-a.s.~for $t\ge0$ thanks to \eqref{def:Xhat}. Since $\eps$ was arbitrary we have $x\mapsto\underline V(x,y)$ non-decreasing.

To prove monotonicity with respect to $y$ we argue in a similar way. We fix $x\in\RR_+$ and $y \le y'$, and for any $\eps>0$ we can find a couple $(\tau_\varepsilon,\gamma_\varepsilon)$ such that 
\begin{align*} 
\underline V(x,y)-\underline V(x,y')\ge& M_{x,y}(\tau_\varepsilon,\gamma_\varepsilon) - M_{x,y'}(\tau_\varepsilon,\gamma_\varepsilon)-\eps \\
=& \EE\left[ e^{-r(\tau_\varepsilon\wedge\gamma_\eps)}(({X}^{x,y}_{\tau_\varepsilon\wedge\gamma_\eps}-K)^+- (X^{x,y'}_{\tau_\varepsilon\wedge\gamma_\eps}-K)^+)\right]-\eps\\ 
\ge&-\eps.
\end{align*}
For the last inequality this time we have used the comparison principle for SDEs, which guarantees ${Y}^y_t \leq {Y}^{y'}_t$, $\PP$-a.s.~for $t\ge0$, and \eqref{def:Xhat}, which gives ${X}^{x,y}_t \ge {X}^{x,y'}_t$, $\PP$-a.s.~for $t\ge0$. By arbitrariness of $\eps$ we obtain the claim.
\vspace{+3pt}

[\emph{Proof of (ii)}] As above we fix $y\in(0,1)$ and $x\ge x'$ so that $\underline{V}(x,y)- \underline{V}(x',y)\ge 0$. For any $\eps>0$ we can find a couple $(\tau_\varepsilon,\gamma_\varepsilon)$ such that 
\begin{align}\label{lipV1} 
0\le \underline{V}(x,y)- \underline{V}(x',y) \leq& M_{x,y}(\tau_\varepsilon,\gamma_\varepsilon) - M_{x',y}(\tau_\varepsilon,\gamma_\varepsilon) +\varepsilon  \nonumber\\
\leq & \EE\left[ e^{-r(\tau_\varepsilon\wedge\gamma_\eps)}|X^{x,y}_{\tau_\varepsilon\wedge\gamma_\eps}- X^{x',y}_{\tau_\varepsilon\wedge\gamma_\eps}|\right] +\varepsilon
\end{align}
where the second inequality uses the Lipschitz property of the call payoff. From \eqref{def:Xhat} we have 
\begin{align*}
e^{-r(\tau_\varepsilon\wedge\gamma_\eps) }|{X}^{x,y}_{\tau_\varepsilon\wedge\gamma_\eps}- {X}^{x',y}_{\tau_\varepsilon\wedge\gamma_\eps}|\le |x-x'| e^{ \sigma W_{\tau_\varepsilon\wedge\gamma_\eps} - \frac{\sigma^2}{2}(\tau_\varepsilon\wedge\gamma_\eps)}.
\end{align*}
Since $\exp(\sigma W_t- \frac{\sigma^2}{2}t)$, $t\ge0$ is a positive supermartingale, 
we deduce that
\begin{align*} 
\EE[e^{-r(\tau_\varepsilon\wedge\gamma_\eps) }|{X}^{x,y}_{\tau_\varepsilon\wedge\gamma_\eps}- {X}^{x',y}_{\tau_\varepsilon\wedge\gamma_\eps}|]\leq |x-x'| 
\end{align*}
and Lipschitz continuity in $x$ follows from \eqref{lipV1} since $\eps>0$ is arbitrary.
\vspace{+3pt}

[\emph{Proof of (iii)}] Now we use the equivalence between the couple $(X^{x,y},Y^y)$ on the space $(\Omega,\CF,\PP)$ and the couple $(S^x,D^y)$ on the space $(\Omega^0,\CF^0,\PP^0_y)$ (see explanation in Sec.~\ref{sec:dyn} and \eqref{gameformulationone} and \eqref{gameformulationtwo}) to write 
\begin{align*}
M_{x,y}(\gamma,\tau)=&\,\EE^0_y\left[ e^{-r\tau}G_1(S^x_{\tau})\indic_{\{\tau \leq \gamma\}} + e^{-r\gamma}G_2(S^x_{\gamma})\indic_{\{ \gamma< \tau\}}\right]\\
=&\,y\,\EE^0_y\left[ e^{-r\tau}G_1(S^x_{\tau})\indic_{\{\tau \leq \gamma\}} + e^{-r\gamma}G_2(S^x_{\gamma})\indic_{\{ \gamma< \tau\}}\big|D=1\right]\\
&+(1-y)\,\EE^0_y\left[ e^{-r\tau}G_1(S^x_{\tau})\indic_{\{\tau \leq \gamma\}} + e^{-r\gamma}G_2(S^x_{\gamma})\indic_{\{ \gamma< \tau\}}\big|D=0\right]
\end{align*}
for any couple $(\tau,\gamma)\in\CT^S$. Set
\begin{align*}
S^{1,x}_t=xe^{\sigma B_t+(r-\delta_0-\frac{\sigma^2}{2})t},\qquad S^{0,x}_t=xe^{\sigma B_t+(r-\frac{\sigma^2}{2})t}
\end{align*}
and notice that conditionally on $D$, the law of $S^x$ is independent of $y$, so denoting $\EE^{W}$ the expectation under the Wiener measure $\WW$ we get 
\begin{align}\label{lipV2}
M_{x,y}(\gamma,\tau)=&\,y\,\EE^W\left[ e^{-r\tau}G_1(S^{1,x}_{\tau})\indic_{\{\tau \leq \gamma\}} + e^{-r\gamma}G_2(S^{1,x}_{\gamma})\indic_{\{ \gamma< \tau\}}\right]\nonumber\\
&+(1-y)\,\EE^W\left[ e^{-r\tau}G_1(S^{0,x}_{\tau})\indic_{\{\tau \leq \gamma\}} + e^{-r\gamma}G_2(S^{0,x}_{\gamma})\indic_{\{ \gamma< \tau\}}\right]
\end{align}

Now we use the above representation of the game payoff as follows. Fix $x\in \RR_+$ and $y\le y'$, then for any $\varepsilon>0$ we find $(\tau_\eps,\gamma_\eps)\in\CT^S$ such that
\begin{align*} 
0\le& \underline{V}(x,y)- \underline{V}(x,y') \\
\leq & M_{x,y}(\gamma_\eps,\tau_\eps) - M_{x,y'}(\gamma_\eps,\tau_\eps) +\varepsilon\\
\le&\,|y-y'|\Big(\EE^W\left[ e^{-r\tau_\eps}G_1(S^{1,x}_{\tau_\eps})\indic_{\{\tau_\eps \leq \gamma_\eps\}} + e^{-r\gamma_\eps}G_2(S^{1,x}_{\gamma_\eps})\indic_{\{ \gamma_\eps< \tau_\eps\}}\right]\\
&\phantom{\,|y-y'|\big(}
+\EE^W\left[ e^{-r\tau_\eps}G_1(S^{0,x}_{\tau_\eps})\indic_{\{\tau_\eps \leq \gamma_\eps\}} + e^{-r\gamma_\eps}G_2(S^{0,x}_{\gamma_\eps})\indic_{\{ \gamma_\eps< \tau_\eps\}}\right]\Big)+\eps.
\end{align*}
For any stopping time $\rho$ and for $k=0,1$ we have $\EE^W[ e^{-r\rho}S^{k,x}_{\rho}] \leq x$ as in \eqref{op-sam}. Moreover $G_1,\, G_2$ have linear growth so that the Lipschitz property of $\underline V(x,\,\cdot\,)$ follows.
\end{proof}

Now we can prove the existence of the value for the game. As explained in the introduction, the main difficulty comes from the fact that we are working with a bi-dimensional stopping game with a lack of uniform integrability on the stopping payoff.

\begin{thm}\label{value}
The game with payoff \eqref{gameformulationtwo} has a value $V(x,y)=\underline V(x,y)=\overline V(x,y)$ for all $(x,y)\in\mathbb R_+\times[0,1]$. Moreover player 2, i.e.~the minimiser (seller), has an optimal strategy 
\begin{align}\label{gamma*}
\gamma_*(x,y)= \inf\{  t\ge 0 \,|\, G_2(X^{x,y}_t) \leq V(X_t^{x,y},Y_t^{x,y})\}
\end{align}
with the convention $\inf \emptyset= +\infty$
and the process  
\[ 
e^{-r(t\wedge \gamma_*)}V(X^{x,y}_{t\wedge \gamma_*},Y^y_{t\wedge \gamma_*}),\quad t\ge 0
\]
is a \textbf{closed} supermartingale.

Finally, if we define 
\begin{align}\label{tau*}
\tau_*(x,y)=\inf \{ t\ge 0 \,|\, V(X^{x,y}_t,Y^y_t) \leq G_1(X^{x,y}_t)\},
\end{align}
with the convention $\inf \emptyset= +\infty$
and the process  
\[
e^{-r(t\wedge\tau_*)}V(X^{x,y}_{t\wedge\tau_*},Y^y_{t\wedge \tau_*}),\quad t\ge 0
\]
 is a (not necessarily closed) submartingale.
\end{thm}

\Proof The proof of Theorem \ref{value} is postponed to the Appendix.

\begin{rems}\label{strictpositive}
According to Lemma \ref{Lipschitz}, the value function is non-increasing with respect to $y$. Therefore for each $y\in(0,1)$, we have 
\begin{align*}
V(x,y)\ge \lim_{z \to 1}V(x,z) \egal V_1(x),
\end{align*}
where $V_1$ is the game value when $\PP(D=1)=1$. According to \cite{yam}, Theorem 2.1, the value function $V_1$ is strictly positive therefore $V$ is also strictly positive.
\end{rems}
 
\section{Properties of the stopping regions}\label{sec:regions}

Having established that the game has a value $V$ we can introduce the so-called continuation region
\begin{align}\label{def:C}
\CC:=\{(x,y)\in\RR_+\times[0,1]\,\tq\,G_1(x)<V(x,y)<G_2(x)\}
\end{align}
and the stopping regions for the two players, i.e.
\begin{equation} \label{region1}
\CS_1=\left\{ (x,y) \in \RR_+\times [0,1]\, \tq\, V(x,y)=G_1(x) \right\},
\end{equation}
for player 1, and
\begin{equation} \label{region2}
\CS_2=\left\{(x,y) \in \RR_+\times [0,1] \,\tq\, V(x,y)=G_2(x) \right\},
\end{equation}  
for player 2. It is clear that $\CC$ is open and $\CS_1$, $\CS_2$ are closed, because $V$ is jointly continuous (see Lemma \ref{Lipschitz}), and obviously $\CS_1\cap\CS_2=\emptyset$. 

These sets are important because, according to theory on zero-sum Dynkin games, the only candidate to be a Nash equilibrium is the pair $(\gamma_*,\tau_*)$ given by \eqref{gamma*} and \eqref{tau*} (see \cite{peskir}). Under complete information the perpetual game call option has been studied in \cite{ekstromvilleneuve} for $y=0$ and in \cite{yam} for $y=1$.
Those papers analyse the geometry of the continuation and stopping regions and for completeness we account for a summary of their results in appendix.  For future reference here we only note that \cite[Sec.~5.1]{ekstromvilleneuve} obtain 
\begin{align}\label{y=0}
\eps_0<K\Leftrightarrow \CS_2\cap\{y=0\}=[K,+\infty).
\end{align}

In the rest of this section we study the shape of the stopping regions. For that we need to introduce the infinitesimal generator of the two-dimensional diffusion $(X,Y)$, i.e.~for any $g\in C^2(\RR_+\times[0,1])$
\begin{align}\label{generateur}
(\CL g)(x,y):=& \Big[(r-\delta_0 y)x \frac{\partial g}{\partial x} + \frac{1}{2} \sigma^2 x^2 \frac{\partial^2 g}{\partial x^2}\\
& \phantom{\Big[}+ \frac{\delta^2_0}{2\sigma^2}y^2(1-y)^2\frac{\partial^2 g}{\partial y^2}-  \delta_0 xy(1-y)\frac{\partial^2 g}{\partial x \partial y}\Big](x,y). \nonumber
\end{align}
 
Let us also introduce the sets 
\begin{align}
A_1&:=\{(x,y)\in (K,+\infty)\times[0,1]\tq (\CL G_1-r G_1)(x,y)>0\}\\
\label{def:A2} A_2&:=\{(x,y)\in (K,+\infty)\times[0,1]\tq (\CL G_2-r G_2)(x,y)<0\}
\end{align}
and notice that indeed $A_1=\{(x,y)\tq xy<rK/\delta_0\:\text{and}\:x > K\}$ and $A_2=\{(x,y)\tq xy>r(K-\varepsilon_0)/\delta_0\:\text{and}\:x > K\}$. We denote the complements of these sets by $A^c_i$, $i=1,2$ and define $\{x>K\}:=(K,+\infty)\times[0,1]$.
\begin{pro}\label{inclusionregions}
We have,
\begin{align*}
\CS_1 \subseteq A^c_1\cap\{x>K\}\quad\text{and}\quad \CS_2\cap \{x>K\} \subseteq A^c_2.
\end{align*}
\end{pro}
  
\begin{proof} It is sufficient to prove the first inclusion (i.e. for $\CS_1$) because arguments for the second one (i.e. for $\CS_2$) are analogous. 

Because $V$ is strictly positive (Remark \ref{strictpositive}), it is clear that $\CS_1 \subset \{x >K\}$. Fix $(x_0,y_0)\in\{x>K\}\cap A_1$, then it is possible to find an open neighbourhood $R$ of $(x_0,y_0)$ such that $R\subset \{x>K\}\cap A_1$, i.e.~$(\CL-r) G_1>0$ on $R$. Let $\tau_R$ be the exit time of $(X^{x_0,y_0},Y^{y_0})$ from $R$ and let $\rho:=\tau_*\wedge \gamma_* \wedge \tau_R$, then Theorem \ref{value} guarantees that 
\[e^{-r(t\wedge\rho)}V(X_{t\wedge\rho},Y_{t\wedge\rho})\quad\text{is a $\PP_{x_0,y_0}$-martingale for $t\ge0$.} \]

Using this property and It\^o's formula we obtain
\begin{eqnarray*}
V(x,y)&=& \EE_{x_0,y_0}\left[e^{-r(t\wedge\rho)}V({X}_{t\wedge\rho},{Y}_{t\wedge\rho})\right]\\
&\ge& \EE_{x_0,y_0}\left[e^{-r(t\wedge \rho)}G_1({X}_{t\wedge \rho})\right]\\
&=& G_1(x_0) + \EE_{x_0,y_0}\left[\int_0^{t\wedge \rho }e^{-rs}(\CL G_1-r G_1)({X}_{s},{Y}_s)\,ds\right]> G_1(x_0),
\end{eqnarray*}
which implies $(x_0,y_0)\notin \CS_1$.
\end{proof}

Our next lemma shows that the stopping region $\CS_1$ is up and right-connected while the region $\CS_2$ is down and left-connected on $\{ x>K\}$.
\begin{lem}\label{monotonic}
The following properties hold
\begin{itemize}
\item[  (i)] $(x,y) \in \CS_1 \;\Rightarrow \; (x,y') \in \CS_1$ for  $y' \geq y$.
\item[ (ii)] $(x,y) \in \CS_2 \;\Rightarrow \; (x,y') \in \CS_2$ for  $y' \leq y$.
\item[(iii)] $(x,y) \in \CS_1 \;\Rightarrow \; (x',y) \in \CS_1$ for  $x' \geq x\geq K$.
\item[ (iv)] $(x,y) \in \CS_2 \;\Rightarrow \; (x',y) \in \CS_2$ for  $x \ge x' \geq K$.
\end{itemize}
\end{lem}
\begin{proof}
The two first properties follow directly from the fact that $y\mapsto V(x,y)$ is non-increasing.  
To prove (iii) let us fix $(x,y)\in \CS_1$ (notice that in particular $x\ge K$). Since $V(x,y)$ is $1$-Lipschitz w.r.t.~$x$ and non-decreasing (see (i)-(ii) in Lemma \ref{Lipschitz}) then for all $x'\geq x$, we have 
\begin{align}\label{right-c}
V(x',y)\leq V(x,y) + (x'-x)=V(x,y) + G_1(x')-G_1(x)=G_1(x'),
\end{align} 
where we have used that $G_1(x')-G_1(x)=x'-x$ for $x \ge x' \geq K$, and that $G_1(x)=V(x,y)$ by assumption. Clearly \eqref{right-c} implies $(x',y) \in \CS_1$ as claimed. Similar arguments give (iv). 
\end{proof}

\begin{lem}\label{rectangle}
For  $x<K$, $V(x,y)<G_2(x)$. Hence $\CS_2\cap [(0,K)\times (0,1)]=\emptyset$.
\end{lem} 
\begin{proof}
Notice that $G_2(x)=\eps_0$ for $(x,y)\in(0,K)\times (0,1)$ and therefore $\CL G_2-r G_2 <0$ on $(0,K)\times (0,1)\subset A_2$. Let $R\subset (0,K)\times (0,1)$ be an open set and fix $(x,y)\in R$. Denote $\rho_R:=\inf\{t\ge0\,\tq\,(X_t,Y_t)\notin R\}$ and let $\tau_*$ be defined by \eqref{tau*}. Notice also that $\tau_*\ge \rho_R$, $\PP$-a.s. because player 1 does not stop in $(0,K)$.
 
Then using Theorem \ref{value} and It\^o formula we obtain
\begin{align*}
V(x,y)\le& \EE\left( e^{-r(t \wedge \rho_R)} V(X^{x,y}_{t \wedge \rho_R}, Y^y_{t \wedge \rho_R})\right)\le  \EE\left( e^{-r(t \wedge \rho_R)} G_2(X^{x,y}_{t \wedge \rho_R})\right)\\
&=  G_2(x) -r\eps_0 \EE \left( \int_0^{t \wedge \rho_R} e^{-r s }\, ds \right)< G_2(x).
\end{align*}
\end{proof}

The next Lemma shows that if the penalty for cancellation does not exceed the strike price, i.e. $\eps_0<K$, then the stopping region $\CS_2$ is non-empty and unbounded.
\begin{lem}\label{S2-non-empty}
If $\varepsilon_0< K$ then the set $\CS_2 \cap [M,+\infty)\times(0,1)$ is non-empty for all $M\geq K$.
\end{lem}
\begin{proof} 
We argue by contradiction and assume that $\CS_2 \cap [M,+\infty)\times(0,1)$ is empty for some $M\ge K$. Fix $(x,y)\in(M,+\infty)\times(0,1)$ and denote $$\rho_M(x,y)=\inf\{ t\geq 0\,|\, X^{x,y}_t \le M\},$$ then clearly $\gamma_* \geq \rho_M$ almost surely. Theorem \ref{value} therefore implies that 
\begin{align*}
t\mapsto e^{-r (t\wedge \rho_M)}V(X^{x,y}_{t\wedge \rho_M},Y^{x,y}_{t\wedge \rho_M})\quad\text{is a supermartingale.}
\end{align*}
For any stopping time $\tau$ we have
\begin{align}\label{ne1}
V(x,y)\geq \EE[ e^{-r \rho_M}V(M,Y^y_ {\rho_M})\indic_{\{\rho_M <\tau\}}+ e^{-r \tau}G_1(X^{x,y}_\tau)\indic_{\{\tau \leq \rho_M\}}].
\end{align}
Using Lipschitz continuity (Lemma \ref{Lipschitz}) and \eqref{y=0}, we also have
\begin{align*}
V(M,y)\ge V(M,0)-C(1+M)y= G_2(M)-C(1+M)y.
\end{align*}
Plugging the latter into \eqref{ne1} to estimate $V(M,Y^y_ {\rho_M})$, recalling $x>M$ and using that $(e^{-rt}Y^y_t)_{t \ge 0}$ is a positive, bounded, supermartingale we obtain
\begin{align*}
V(x,y)&\geq \EE[ e^{-r \rho_M}G_2(M)\indic_{\{\rho_M <\tau\}}+e^{-r \tau}G_1(X^{x,y}_\tau)\indic_{\{\tau \leq \rho_M\}}] - C(1+x)y. 
\end{align*}

Since $\tau$ was arbitrary we then have $V(x,y)\ge f_M(x,y)-C(1+x)y$ where 
\begin{align*}
f_M(x,y):=\sup_{\tau} \EE[ e^{-r \rho_M}G_2(M)\indic_{\{\rho_M <\tau\}}+e^{-r \tau}G_1(X^{x,y}_\tau)\indic_{\{\tau \leq \rho_M\}}].
\end{align*}
The same arguments as in the proof of Lemma \ref{Lipschitz} allow us to prove that 
\begin{align*}
|f_M(x,y)-f_M(x,y')|\le C(1+x)|y-y'|\quad\text{for all $y,y'\in[0,1]$ and $x\in\RR_+$}. 
\end{align*}
We can now use the above to obtain
\begin{align}\label{ne2}
V(x,y)&\ge f_M(x,y)- C(1+x)y\ge f_M(x,0)-2C(1+x)y.
\end{align}

Next we want to find a lower bound for $f_M(x,0)$. Notice that for $t\ge 0$
\begin{align*}
Y^0_t=0\quad\text{ and }\quad X^{x,0}_t=xe^{\sigma W+(r-\sigma^2/2)t},\:\:\PP-\text{a.s.}
\end{align*}
For $n\ge x$, setting 
\begin{align*}
\varphi(x):=x^{-\frac{2r}{\sigma^2}},
\quad \psi(x):=x \quad \text{and} \quad \tau_n:=\inf\{t\ge0 \tq X^{x,0}\ge n\}
\end{align*}
we can rely on standard formulae for the Laplace transform of $\rho_M$ and $\tau_n$ to obtain 
\begin{align*}
f_M(x,0)\ge& G_2(M)\EE\left[e^{-r\rho_M}\indic_{\{\rho_M<\tau_n\}}\right]+G_1(n)\EE\left[e^{-r\tau_n}\indic_{\{\rho_M\ge\tau_n\}}\right]\\
=&G_2(M)\frac{\psi(x)\varphi(n)-\varphi(x)\psi(n)}{\psi(M)\varphi(n)-\varphi(M)\psi(n)}+G_1(n)\frac{\psi(M)\varphi(x)-\varphi(M)\psi(x)}{\psi(M)\varphi(n)-\varphi(M)\psi(n)}.
\end{align*}
Letting $n\to\infty$ it is easy to check that 
\[  f_M(x,0) \geq x -(K-\varepsilon_0) \left(\frac{M}{x}\right)^{\frac{2r}{\sigma^2}} > G_2(x)\]
where the final inequality uses $x>M$. The latter and \eqref{ne2} imply that $V(x,y)>G_2(x)$ for $y$ sufficiently small, and thus a contradiction.
\end{proof}

Thanks to above lemmas we can define boundaries of the stopping regions as follows
\begin{align}\label{def:b1b2}
b_1(y)&:= \inf \{ x \in [0,+\infty)\, \tq\, V(x,y) = G_1(x)\},\\
b_2(y)&:= \sup \{ x \in [0,+\infty)\, \tq\, V(x,y) = G_2(x)\},
\end{align}
with the usual convention that $\inf\emptyset=+\infty$ and $\sup\emptyset=0$. Notice that $\CS_2\cap(\RR_+\times\{y\})=[K,b_2(y)]$ if $b_2(y)\ge K$ and it is empty otherwise. From Lemma \ref{monotonic} and because the sets $\CS_i$ are closed, we deduce the next corollary
\begin{cor}\label{cor:b}
The functions $b_1$ and $b_2$ are non-increasing on their respective domains and determine the stopping sets as follows:
\[
\CS_1=\{(x,y)\in\RR_+\times[0,1]: x\ge b_1(y)\},\quad \CS_2=\{(x,y)\in\RR_+\times[0,1]: K\le x\le b_2(y)\}.
\]

 Moreover $b_1$ is lower semi-continuous (hence right-continuous) whereas $b_2$ is upper-semi-continuous (hence left-continuous). Finally, thanks to Proposition \ref{inclusionregions} and the definition of $A_1$ and $A_2$ we have $b_1 \ge b_2$ on $[0,1]$.
\end{cor}
Next we show $b_1$ is a well-defined function on $(0,1)$.
\begin{lem}\label{lem:bfinite}
For all $y\in(0,1)$, $b_1(y) <\infty$.
\end{lem} 
\begin{proof}
Arguing by contradiction let us assume that there exists $y_0\in(0,1)$ such that $b_1(y_0)=+\infty$. Then by monotonicity of $b_1$ ((i) and (iii) of Lemma \ref{monotonic}) and lower semi continuity, it holds $b_1(y)=+\infty$ on $[0,y_0]$. 

Denote $\rho_{0}:=\inf\{ t\ge 0\, \tq\, Y_t \ge y_0\}$. We thus have $\rho_{0} \le \tau_*$, $\PP_{x,y}$-a.s.~for any starting point $(x,y)$ with $y\in (0,y_0)$. From now on fix $y\in(0,y_0)$. Theorem \ref{value} guarantees that 
\begin{align*}
t\mapsto e^{-r (t\wedge \rho_{0})}V(X^{x,y}_{t\wedge \rho_{0}},Y^{y}_{t\wedge \rho_{0}})\quad\text{is a submartingale.}
\end{align*}
Therefore, using also that $V\le G_2$, for any $t>0$ we have   
\begin{align*} 
V(x,y)  \leq& \EE\left[ e^{-r(\rho_0 \wedge t)}V(X^{x,y}_{\rho_0 \wedge t}, Y^y_{\rho_0 \wedge t})\right]
\leq \EE\left[e^{-r(\rho_0 \wedge t)} X^{x,y}_{\rho_0 \wedge t} \right] +\varepsilon_0=\alpha(t,y) x +\varepsilon_0 
\end{align*}
with 
\begin{align*}
\alpha(t,y):=\EE\left[ e^{\sigma W_{\rho_0 \wedge t} - \frac{\sigma^2}{2}(\rho_0 \wedge t)}e^{-\delta_0 \int_0^{\rho_0 \wedge t}Y^y_sds }\right]<\EE\left[ e^{\sigma W_{\rho_0 \wedge t} - \frac{\sigma^2}{2}(\rho_0 \wedge t)}\right]=1.
\end{align*}
According to the last two expressions above, for fixed $y\in(0,y_0)$, we get
\begin{align*}
\lim_{x\to\infty}x^{-1}V(x,y)=\alpha(t,y)<1=\lim_{x\to\infty}x^{-1}G_1(x)
\end{align*}
which contradicts $V\ge G_1$.   
\end{proof}

\begin{lem}
If $\varepsilon_0 < K$ then 
\begin{align*}
\displaystyle{\lim_{y  \to 0}} b_2(y)=+\infty,\qquad \lim_{y\rightarrow 0+}\;b_1(y)=+\infty.
\end{align*}
\end{lem} 
\begin{proof}
From Corollary \ref{cor:b} we have $b_1(y)\ge b_2(y)$ for all $y\in(0,1)$ . Since Lemma \ref{S2-non-empty} holds, then it must be $\lim_{y\to0}b_2(y)=+\infty$. The latter also gives $\lim_{y\to0}b_1(y)=+\infty$. 

\end{proof}

From now on, whenever we refer to properties of $\CS_2$ and its boundary, we tacitly assume that $\CS_2\cap(\mathbb{R}_+\times(0,1))\neq \emptyset$. We recall that indeed this is always true for $\eps_0<K$, thanks to Lemma \ref{S2-non-empty}. In this context we also denote
\begin{align}\label{b2K} 
b^K_2:=\sup\{y>0\,\tq\,(K,y)\in\CS_2\},
\end{align}
and notice that Lemma \ref{S2-non-empty} (with $M=K$) implies that the set $\{y>0\,\tq\,(K,y)\in\CS_2\}$ is non-empty.

\section{A parabolic formulation of the problem}

In order to study existence of Nash equilibria and regularity of the value function of the game (beyond continuity) it is useful to introduce a deterministic transformation of the process $(X,Y)$. Such transformation also unveils a parabolic nature of the problem. 

Given $(x,y)\in (0,\infty)\times (0,1)$, let us define $z=\ln(x)+ \frac{\sigma^2}{\delta_0}\ln(\frac{y}{1-y})$ and the process $Z^z$ such that $Z^z_0=z$ and:
\begin{align}\label{transf01}
Z^z_t=\ln(X^{x,y}_t) +  \frac{\sigma^2}{\delta_0} \ln\left(\frac{Y^y_t}{1-Y^y_t}\right).
\end{align}

Then setting 
\begin{align}\label{def:k}
k:=(r - \tfrac{\sigma^2}{2}- \tfrac{\delta_0}{2})
\end{align}
it is not hard to check, by using It\^o's formula, that $Z^z$ evolves according to 
\begin{align}\label{Z}
Z^z_t=z+k t,\qquad t\ge 0.
\end{align}

From \eqref{transf01} we observe that $\PP$-almost surely
\begin{align}
X^{x,y}_t=F(Z^z_t,Y^y_t),\qquad t\ge0
\end{align}
with $F:\RR\times (0,1)\to\RR_+$ defined by
\begin{align}\label{def:F}
F(z,y)= \exp \left(z-\frac{\sigma^2}{\delta_0}\ln\left(\frac{y}{1-y}\right)\right)= e^z \left(\frac{1-y}{y} \right)^{\frac{\sigma^2}{\delta_0}}.
\end{align}
Notice that $F$ is $C^2$ on $\RR\times (0,1)$.
The process $Z$ is indeed deterministic and of bounded variation, hence it plays the role of a ``time'' process. Whether $Z$ is increasing or decreasing depends on the sign of $k$. In the rest of the paper we study the case $k\neq 0$ which is \emph{truly} two-dimensional. We leave aside the case $k=0$ that reduces to a one-dimensional problem parametrised in the variable $z$.

\begin{rem}\label{rem:supp}
In the new coordinates it becomes clear that the law of $(X^{x,y}_t,Y^y_t)$ is supported on the curve $\{(F(z+kt,\zeta),\zeta),\, \zeta\in(0,1)\}$, which is a set of null Lebesgue measure in $\RR_+\times[0,1]$.
\end{rem}

We can now look at our game in the new coordinates and consider the functions $H_1,H_2,v:\RR\times (0,1)\to\RR_+$ given by
\begin{align}\label{vH}
v(z,y):=V(F(z,y),y),\quad H_1(z,y):=G_1(F(z,y)),\quad H_2(z,y):=G_2(F(z,y)).
\end{align}
By construction, we have $H_1 \le v \le H_2$ and $v$ is equal to the value of the stopping game
\begin{align*}
v(z,y)=\sup_{\tau \in \CT}\; \inf_{\gamma \in \CT}\; \EE\left[ e^{-r\tau}H_1(Z^z_\tau , Y^y_{\tau})\indic_{\{\tau \leq \gamma\}}+  e^{-r\gamma} H_2(Z^z_\gamma , Y^y_{\gamma})\indic_{\{\gamma < \tau\}}\right]
\end{align*}
For this new parametrization of the game we naturally introduce the continuation and stopping regions
\begin{align*}
\CC':=&\{(z,y)\in \RR\times (0,1)\,\tq\,H_1(z,y)<v(z,y)<H_2(z,y)\}\,\\
\CS'_1:=&\{ (z,y)\in \RR\times (0,1)\,\tq\, v(z,y)=H_1(z,y) \},\\
\CS'_2:=&\{ (z,y)\in \RR\times (0,1)\,\tq\, v(z,y)=H_2(z,y) \}.
\end{align*}
Using Lemma \ref{Lipschitz} it is immediate to verify that $v$ is locally Lipschitz continuous in $\RR\times(0,1)$ so that $\CC'$ is open and $\CS'_i$, $i=1,2$ are closed. Moreover, it is clear that $\gamma_*$ and $\tau_*$ as in \eqref{gamma*}-\eqref{tau*} are the entry times of $(Z,Y)$ into $\CS_2'$ and $\CS_1'$, respectively.

The infinitesimal generator associated with $(Z,Y)$ is defined by
\begin{align}
(\CG\,f)(z,y):=k\frac{\partial\, f}{\partial\,z }(z,y)+\frac{1}{2}\left(\frac{\delta_0}{\sigma}\right)^2y^2(1-y)^2\frac{\partial^2 f}{\partial \,y^2}(z,y),
\end{align}
for $f\in C^{1,2}(\RR\times[0,1])$.
One advantage of this formulation is that $\CG$ is a parabolic operator and it is non-degenerate on $\RR\times(0,1)$, so that the associated Cauchy-Dirichlet problems admit classical solutions under standard assumptions on the boundary conditions.

Since $v$ is continuous then $\{e^{-rt}v(Z_t,Y_t),\,t\le\tau'_{\CC}\}$ is a continuous martingale for $\tau'_{\CC}:=\inf\{t\ge 0\,\tq\,(Z_t,Y_t)\notin\CC'\}$ (the latter follows from Theorem \ref{value} and the fact that $(X,Y)$ is linked to $(Z,Y)$ by a deterministic map). We can use results of interior regularity for solutions to parabolic PDEs (see, e.g.,~\cite[Corollary 2.4.3]{krylov}) and It\^o's formula to deduce that any solution $f$ to
$
(\CG f-r f)(z,y)=0 \hbox{ on } R_\eta=(z_0,z_0+\eta)\times (y_0-\eta,y_0+\eta) \subset \CC'
$
with $f=v$ on $\partial R_\eta$ is $C^{\infty}(R_\eta)$ and coincides with $v$. Therefore,
 $v\in C^{\infty}(\CC')$ and thus it satisfies
\begin{align}
\label{freeb1}&(\CG v-r v)(z,y)=0,\qquad\text{for $(z,y)\in \CC'$},\\[+4pt]
\label{freeb2}&H_1\le v\le H_2, \qquad\qquad\text{on $\RR\times(0,1)$,}\\[+4pt]
\label{freeb3}&v|_{\partial \CS'_1}=H_1|_{\partial \CS'_1}\quad\text{and}\quad v|_{\partial \CS'_2}=H_2|_{\partial \CS'_2}.
\end{align}
\noindent As a consequence, $V \in C^{\infty}(\CC)$ as well.

We denote by $R_K$ the closure in $\RR\times (0,1)$ of the set in which $H_1>0$, i.e.
\begin{align}\label{RK}
R_K=&\{ (z,y) \in \RR\times (0,1)\,\tq\, F(z,y) \ge K\}
=\{(z,y)\in\RR\times(0,1)\,\tq\,y\le y_K(z)\}
\end{align}
where $y_K(z):=e^{\delta_0/\sigma^2\,z}/(K^{\delta_0/\sigma^2}+e^{\delta_0/\sigma^2\,z})$. 
According to Proposition \ref{inclusionregions} and Lemma \ref{rectangle}, the stopping regions $\CS'_1$ and $\CS'_2$ lie in $R_K$. Notice that since $y_K$ is increasing, if $(z_0,y_0) \in R_K$ then any pair $(z,y_0)$ belongs to $R_K$ for $z \ge z_0$.
Somewhat in analogy with \eqref{b2K} we also define
\begin{align}\label{zK}
z_K:=\sup\{z\in\RR\,\tq\,(z,y_K(z))\in\CS'_2\}\:\:\text{and}\:\: \overline y_K:=y_K(z_K).
\end{align}

In the new coordinates the sets $\CS'_1$ and $\CS'_2$ are connected with respect to the $z$ variable, as illustrated in the next lemma.
\begin{lem}\label{connectedness}
Let $(z_0,y_0) \in R_K$.
\begin{itemize}
\item[ (i)]  $(z_0,y_0) \in \CS'_2\implies (z,y_0) \in \CS'_2$ for all $z \le z_0$, such that $(z,y_0)\in R_K$,
\item[(ii)]  $(z_0,y_0) \in \CS'_1\implies (z,y_0) \in \CS'_1$ for all $z \geq z_0$.
\end{itemize}
\end{lem}
\begin{proof}
Using that $z\mapsto F(z,y)$ is increasing for each $y\in(0,1)$ it is not difficult to show (by direct comparison) that $z\to v(z,y)$ is also non-decreasing. To prove $(i)$ take $z\le z_0$, then $(ii)$ of Lemma \ref{Lipschitz} implies 
\begin{align}\label{bWt}
0&\le v(z_0,y_0)-v(z,y_0)=V(F(z_0,y_0),y_0)-V(F(z,y_0),y_0)\\
&\le F(z_0,y_0)-F(z,y_0)=H_2(z_0,y_0)-H_2(z,y_0).\nonumber
\end{align}
If $(z_0,y_0) \in \CS'_2$ then $v(z_0,y_0)-H_2(z_0,y_0)= 0$ yielding $v(z,y_0)-H_2(z,y_0)\ge 0$. With an analogous argument we can prove $(ii)$.
\end{proof}

The stopping sets are not necessarily connected with respect to the $y$ variable and indeed we only have connected sets for some values of the rate $\delta_0$ and volatility $\sigma$ of $X$. In particular in the rest of the paper we make the following standing assumption (unless otherwise specified).
\begin{ass}\label{ass:2}
We assume $\tfrac{\sigma^2}{\delta_0}\ge 1$.
\end{ass}
\noindent For $\tfrac{\sigma^2}{\delta_0}\ge 1$ the sets $\CS'_1$, $\CS'_2$ enjoy the next desired property.
\begin{lem}\label{connectednessagain}
Let $(z_0,y_0) \in R_K$, then 
\begin{itemize}
\item[ (i)] $(z_0,y_0) \in \CS'_2\implies (z_0,y) \in \CS'_2$ for all $y \ge y_0$ such that $(z_0,y)\in R_K$,
\item[(ii)] $(z_0,y_0) \in \CS'_1\implies (z_0,y) \in \CS'_1$ for all $y \le y_0$. 
\end{itemize}
Moreover it also holds
\begin{itemize}
\item[(iii)] $v(z,y')\le v(z,y)$ for $y'\ge y>y_K(z)$, $z\in\RR$.
\end{itemize}
\end{lem}
\begin{proof}
Take $(x_0,y_0)\in R_K$, fix $y\ge y_0$ and let $x_0=F(z_0,y_0)$. Let $\gamma:=\gamma_*(x_0,y)$ be optimal for player 2 in the game started at $(x_0,y)$ and $\tau$ an $\eps$-optimal stopping time for player 1 in the game started at $(x_0,y_0)$. Recall also that on $\{\tau\wedge\gamma=+\infty\}$ both players have zero payoff due to \eqref{infty}. Then using that $H_2(z,y)-H_2(z,y')=H_1(z,y)-H_1(z,y')$ for all $z\in\RR$ and $y,y'\in(0,1)$ we obtain
\begin{align}\label{trans2}
v(z_0,y_0)& - v(z_0,y) \\
\leq& \EE\left[e^{-r\gamma}H_2(Z^{z_0}_{\gamma},Y^{y_0}_{\gamma})\indic_{\{\gamma < \tau\}}+ 
e^{-r \tau}H_1(Z^{z_0}_{\tau},Y^{y_0}_{\tau})\indic_{\{\tau \leq \gamma\}}\right]\nonumber\\
& -\EE\left[e^{-r\gamma}H_2(Z^{z_0}_{\gamma},Y^{y}_{\gamma})\indic_{\{\gamma < \tau\}}+ e^{-r \tau}H_1(Z^{z_0}_{\tau},Y^{y}_{\tau})\indic_{\{\tau \leq \gamma\}}\right]+\eps\nonumber\\
=&\EE\left[e^{-r(\gamma\wedge\tau)}\left(H_1(Z^{z_0}_{\gamma\wedge\tau},Y^{y_0}_{\gamma\wedge\tau})-H_1(Z^{z_0}_{\gamma\wedge\tau},Y^{y}_{\gamma\wedge\tau})\right)\right]+\eps.\nonumber
\end{align}

Now we notice that, since $Y^{y}_t\ge Y^{y_0}_t$, $\PP$-a.s.~for all $t\ge0$ and $y\mapsto F(z,y)$ is decreasing, then  
\begin{align}\label{ine0}
H_1(Z^{z_0}_{\gamma\wedge\tau},Y^{y_0}_{\gamma\wedge\tau})-H_1(Z^{z_0}_{\gamma\wedge\tau},Y^{y}_{\gamma\wedge\tau})\le F(Z^{z_0}_{\gamma\wedge\tau},Y^{y_0}_{\gamma\wedge\tau})-F(Z^{z_0}_{\gamma\wedge\tau},Y^{y}_{\gamma\wedge\tau})
\end{align}
and the right-hand side of the inequality is positive. Therefore we can use Fatou's lemma and \eqref{ine0} to obtain 
\begin{align}\label{ine1}
v(z_0,y_0) - v(z_0,y)\le& \EE\left[e^{-r(\gamma\wedge\tau)}\left(F(Z^{z_0}_{\gamma\wedge\tau}, Y^{y_0}_{\gamma\wedge\tau})- F(Z^{z_0}_{\gamma\wedge\tau}, Y^{y}_{\gamma\wedge\tau})\right)\right]+\eps \nonumber\\
\le&\liminf_{t\to +\infty}\EE\left[e^{-r(\gamma\wedge\tau\wedge t)}\left(F(Z^{z_0}_{\gamma\wedge\tau\wedge t}, Y^{y_0}_{\gamma\wedge\tau\wedge t})- F(Z^{z_0}_{\gamma\wedge\tau\wedge t}, Y^{y}_{\gamma\wedge\tau\wedge t})\right)\right]+\eps
\end{align}
Setting $M_s^\zeta=e^{-rs}F(Z^{z_0}_s,Y^{\zeta}_s)$, for any $\zeta\in(0,1)$ and $s\in[0,t]$, It\^o formula gives
\[ dM^\zeta_s= -\delta_0 Y^\zeta_s M_s^\zeta dt + \sigma M_s^\zeta dW_s .\]
Hence substituting the above into \eqref{ine1} and noticing that $M^\zeta\in L^2([0,t]\times\Omega)$, we can use the optional sampling theorem to obtain
\begin{align}\label{trans3}
v(z_0,y_0) - v(z_0,y)\le&F(z_0,y_0)-F(z_0,y)+\eps\nonumber\\
&+\delta_0\liminf_{t\to\infty}\EE\left[\int_0^{\gamma\wedge \tau\wedge t}\left(Y^{y}_sF(Z^{z_0}_s,Y^y_s)-Y^{y_0}_sF(Z^{z_0}_s,Y^{y_0}_s)\right)dt\right].
\end{align}
Using now that, for $\frac{\sigma^2}{\delta_0}\geq 1$, the map $y \rightarrow yF(z,y)$ is non-increasing with
\begin{align}\label{h}
\frac{\partial}{\partial y} \left(yF(z,y)\right)= e^z  \left(\frac{1-y}{y} \right)^{\frac{\sigma^2}{\delta_0}} \left(1- \frac{\sigma^2/\delta_0}{(1-y)}\right),
\end{align}
and recalling once again that $Y^y_{\cdot}\ge Y^{y_0}_{\cdot}$, we see that \eqref{trans3} implies
\begin{align}\label{trans4}
v(z_0,y_0) - v(z_0,y)\le &F(z_0,y_0)-F(z_0,y)+\eps.
\end{align}
Since $\eps$ is arbitrary then $y\mapsto (v(z,y)-F(z,y))$ is non-decreasing and therefore $(i)$ and $(ii)$ easily follow.

The proof of $(iii)$ follows from the fact that $y\mapsto H_i(z,y)$ is decreasing for $i=1,2$.
\end{proof}

The next corollary is a simple consequence of Lemma \ref{connectedness} and \ref{connectednessagain}. We recall that $\CS_i\cap R^c_K=\emptyset$ as no player stops for $X_t<K$ (see Remark \ref{strictpositive} and Lemma \ref{rectangle}).

\begin{figure}[t]
\includegraphics[width=0.9\textwidth]{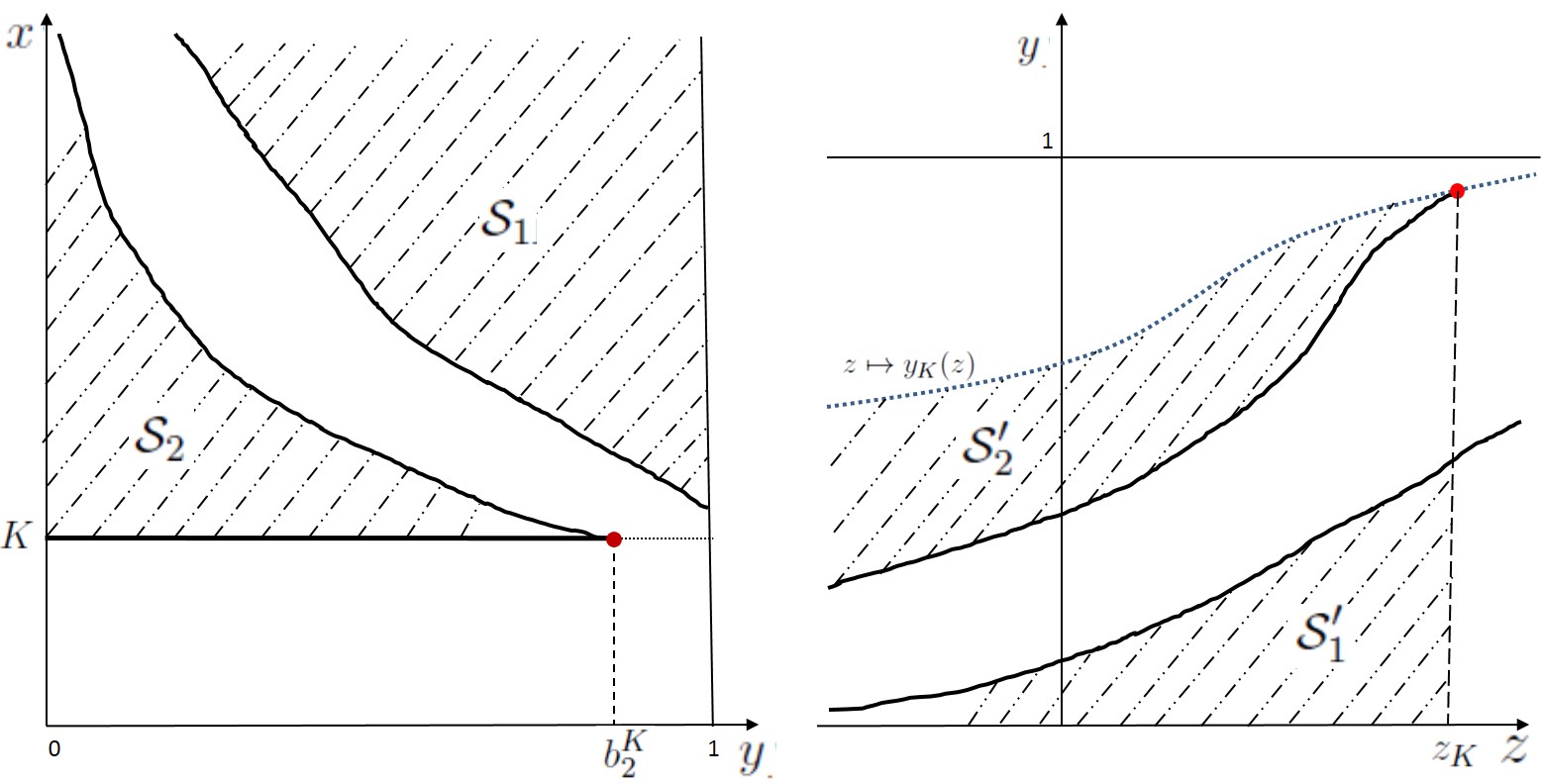}
\caption{An illustration of the sets $\CS_1$, $\CS_2$ (left) and of the sets $\CS_1'$, $\CS_2'$ (right).}
\end{figure}

\begin{cor}\label{cor:bb}
There exists non-decreasing functions $c_{1}:\RR\to[0,1]$, $c_{2}:(-\infty,z_K]\to[0,1]$, with $c_1(\cdot)\le c_2(\cdot)\le y_K(\cdot)$ on $(-\infty,z_K]$ and $c_1(\cdot)\le y_K(\cdot)$ on $(z_K,+\infty)$, such that
\begin{align}
&\CS'_1=\{(z,y)\in \RR\times[0,1]\,\tq\,y\le c_1(z)\},\\
&\CS'_2=\{(z,y)\in (-\infty,z_K]\times[0,1]\,\tq\,y\in[ c_2(z),y_K(z)]\}.
\end{align} 
\end{cor}

Next we provide continuity of the boundaries $b_i$ and $c_i$, $i=1,2$.

\begin{pro}\label{b2continuous}
The stopping boundaries $b_1,b_2$ and $c_1,c_2$ are continuous.
\end{pro}
\begin{proof}
\emph{Step 1.} First we prove the claim for $b_1,b_2$.
Since the proofs are similar for the two boundaries, we only provide details for $b_2$.
According to Corollary \ref{cor:b}, the boundary $b_2$ is left-continuous. To show the right-continuity, we will argue by contradiction.

Assume that there exists $y_0 \in (0,1)$ such that $b_2(y_0+) < b_2(y_0)$ and fix $x_0 \in (b_2(y_0+),b_2(y_0))$. Next define $z_0$ by $F(z_0,y_0)=x_0$ and notice that since $x_0 < b_2(y_0)$, then $(F(z_0,y_0),y_0) \in \CS_2$ and therefore $(z_0,y_0) \in \CS'_2$. We take a decreasing sequence $(y_n)_n$ with $y_n\downarrow y_0$ as $n\to\infty$ so that Lemma \ref{connectednessagain} implies that $(z_0,y_n) \in \CS'_2$ for all $n$. Equivalently $x_n=F(z_0,y_n) \le b_2(y_n)$ so that taking limits and using that $F$ is continuous, we obtain $x_0=F(z_0,y_0)\le b(y_0+)$. The latter is a contradiction.  
\vspace{+4pt}

\emph{Step 2.} Now we show continuity of $c_1,c_2$. Let us start from $c_2$ and fix $z_0$. Take a sequence $z_n\uparrow z_0$ as $n\to\infty$ so that $(z_n,c_2(z_n))\to (z_0,c_2(z_0-))$, where $c_2(z_0-)\le c_2(z_0)$ and the limit exists by monotonicity. Since $\CS_2'$ is closed we have $(z_0,c_2(z_0-))\in\CS_2'$ and therefore $c_2(z_0-)\ge c_2(z_0)$, hence implying left-continuity. 

To prove that $c_2$ is also right-continuous we use Theorem 3.3 in \cite{DeA15}. Since the latter theorem is not given in our game context we repeat here some arguments for completeness. Let us assume $c_2(z_0)<c_2(z_0+)$ and denote $y_0:=c_2(z_0)$, $y_1:=c_2(z_0+)$ for simplicity. Fix $z_1>z_0$ such that the open rectangle $\CR$ with vertices $(z_0,y_0)$, $(z_0,y_1)$, $(z_1,y_1)$ and $(z_1,y_0)$ is contained in $\CC'$. Let $w:=v-H_2$ and $\overline w:=\partial w/\partial y$, then results of interior regularity for solutions to PDEs (see e.g. \cite[Corollary 2.4.3]{krylov}) imply that $\overline w\in C^{1,2}(\CR)$ and, by deriving \eqref{freeb1} with respect to $y$, it turns out that 
\begin{align}\label{cont1}
\left(k\frac{\partial \overline w}{\partial z}+\CA\overline w -r\overline w\right)(z,y) = \delta_0 h(z,y),\quad\text{for $(z,y)\in\CR$,}
\end{align}
where $h(z,y):=\tfrac{\partial}{\partial y}(y F(z,y))$ and 
\begin{align*}
\CA:= \tfrac{1}{2}\left(\frac{\delta_0}{\sigma}\right)^2y^2(1-y)^2\frac{\partial^2}{\partial y^2}+\left(\frac{\delta_0}{\sigma}\right)^2y(1-y)(1-2y)\frac{\partial}{\partial y}.
\end{align*}

Let $\psi\in C^\infty_c(y_0,y_1)$ be positive and such that $\int^{y_1}_{y_0}\psi(y)dy=1$. Multiply \eqref{cont1} by $\psi$ and integrate by parts over $(y_0,y_1)$ to obtain
\begin{align}
L_\psi(z):=&k\int^{y_1}_{y_0}\frac{\partial\overline w}{\partial z}(z,y)\psi(y)dy\\
=&\int^{y_1}_{y_0}w(z,y)\frac{\partial}{\partial y}[(r-\CA^*)]\psi(y)dy+\delta_0\int^{y_1}_{y_0}h(z,y)\psi(y)dy,\nonumber
\end{align}
where $\CA^*$ is the adjoint operator of $\CA$. Now, taking limits as $z\downarrow z_0$, we can use dominated convergence in the right hand side of the above equation and the fact that $w(z_0,\cdot)=0$ on $(y_0,y_1)$, to find
\begin{align}\label{cont2}
L_\psi(z_0+)=\lim_{z\downarrow z_0}L_\psi(z)=\delta_0\int^{y_1}_{y_0}h(z_0,y)\psi(y)dy\le -\delta_0\ell. 
\end{align}
In the final inequality we set $-\ell:=\sup_{(y_0,y_1)}h(z_0,y)$ and notice that $\ell>0$ due to $\sigma^2/\delta_0\ge 1$ (see \eqref{h}).

By its definition $L_\psi\in C(z_0,z_1)$ and \eqref{cont2} implies that its right limit at $z_0$ exists and it is strictly negative. Then for some $\delta>0$, using integration by parts and Fubini's theorem, we have
\begin{align}
0>&\int^{z_0+\delta}_{z_0}L_\psi(z)dz=-k \int^{y_1}_{y_0}\left(\int^{z_0+\delta}_{z_0}\frac{\partial w}{\partial z}(z,y)dz\right)\psi'(y) dy\nonumber\\
=&-k\int^{y_1}_{y_0}w(z_0+\delta,y)\psi'(y)dy= k\int^{y_1}_{y_0}\frac{\partial w}{\partial y}(z_0+\delta,y)\psi(y)dy\ge 0,
\end{align} 
where the last inequality follows because $w(z_0+\delta,\cdot)$ is non-decreasing as shown in the proof of Lemma \ref{connectednessagain}. Therefore we reach a contradiction and $c_2$ must be continuous at $z_0$. By arbitrariness of $z_0$ we conclude that $c_2$ is continuous.

To prove continuity of $c_1$ we simply refer to \cite[Thm.~3.1]{DeA15}. The latter is not obtained in a game context but arguments as above allow a straightforward extension to it. We also notice that in applying that theorem we use that $v$ is locally Lipschitz on $\mathbb{R}\times(0,1)$.
\end{proof}

\section{Regularity across the boundaries}

In this section we show that the value function $V$ is indeed $C^1$ in $\RR_+^*\times(0,1)$. The key to this result is the so-called \emph{regularity} of the optimal boundaries. Roughly speaking this means that the process $(X,Y)$ immediately enters the interior of the sets $\CS_1$ and $\CS_2$ upon hitting their boundaries $\partial \CS_1$ and $\partial \CS_2$. Analogous considerations apply to the process $(Z,Y)$ and the sets $\CS'_1$, $\CS'_2$.

We recall that we work under Assumption \ref{ass:2}.
Let us introduce the hitting times 
\begin{align}
&\hat{\tau}_*:=\inf\{t>0\,\vert\,(X_t,Y_t)\in \CS_1\}=\inf\{t>0\,\vert\,(Z_t,Y_t)\in \CS'_1\}\\[+4pt]
&\hat{\gamma}_*:=\inf\{t>0\,\vert\,(X_t,Y_t)\in \CS_2\}=\inf\{t>0\,\vert\,(Z_t,Y_t)\in \CS'_2\}.
\end{align}
The next lemma provides a clear statement of the regularity of the optimal boundaries for the diffusions $(X,Y)$ and $(Z,Y)$. Its proof is postponed to the end of the section so that we can move quickly towards the main result, i.e.~Proposition \ref{prop:C1}. 
\begin{lem}\label{lem:reg-b}
If $(x_0,y_0)\in\partial \CS_1$ (resp.~$(z_0,y_0)\in\partial \CS'_1$) then 
\begin{align}\label{regS1}
\PP_{x_0,y_0}(\hat{\tau}_*>0)=0\qquad(\text{resp.}~\PP_{z_0,y_0}(\hat{\tau}_*>0)=0).
\end{align}
Similarly, if $(x_0,y_0)\in\partial \CS_2$ (resp.~$(z_0,y_0)\in\partial \CS'_2$) then
\begin{align}\label{regS2}
\PP_{x_0,y_0}(\hat{\gamma}_*>0)=0\qquad(\text{resp.}~\PP_{z_0,y_0}(\hat{\gamma}_*>0)=0).
\end{align}
Notice that if $k>0$ \eqref{regS2} holds with $(x_0,y_0)\neq(K,b_2^K)$ (resp.~$(z_0,y_0)\neq(z_K,\overline y_K)$).
\end{lem}
Adopting the convention that $[K,b_2(y)]=\emptyset$ for $y>\overline y_K$ and $[c_2(z),y_K(z)]=\emptyset$ for $z>z_K$, we can use Corollary \ref{cor:bb} and write $\PP$-a.s.
\begin{align}
&\hat{\tau}_*=\inf\{t>0\,\vert\,X_t\ge b_1(Y_t)\}=\inf\{t>0\,\vert\,Y_t\le c_1(Z_t)\}\\[+4pt]
&\hat{\gamma}_*=\inf\{t>0\,\vert\,X_t\in[K,b_2(Y_t)]\}=\inf\{t>0\,\vert\,Y_t\in[c_2(Z_t),y_K(Z_t)]\}. 
\end{align}
To avoid further technicalities we assume that 
\[
c_2(z)\neq y_K(z)~\text{for $z<z_K$}\quad\text{(resp.~}b_2(y)\neq K~\text{for $y<b^K_2$)},
\]
however all the results of this section can be easily adapted to the case in which $c_2=y_K$ for some $z$ (i.e.~$b_2=K$ for some $y$).

We consider hitting times to the interior of the stopping sets, i.e.~we define $\PP$-a.s.
\begin{align}
&\check{\tau}:=\inf\{t>0\,\vert\,X_t> b_1(Y_t)\}=\inf\{t>0\,\vert\,Y_t< c_1(Z_t)\}\\[+4pt]
&\check{\gamma}:=\inf\{t>0\,\vert\,X_t\in(K, b_2(Y_t))\}=\inf\{t>0\,\vert\,Y_t\in( c_2(Z_t),y_K(Z_t))\}.
\end{align}
Notice that for each line, the second equality follows from the continuity of the optimal boundaries. Precisely, for all $(z,y) \in R_K$, we have the equivalences 

\[ 
F(z,y)< b_2(y) \Leftrightarrow y > c_2(z),\:\:\text{and}\:\: F(z,y)> b_1(y) \Leftrightarrow y < c_1(z).
\]
We remark that if $c_2=y_K$ on an interval $\mathcal I$ then $\check\gamma$ should account also for the first \emph{crossing} time of $c_2\,|_{\mathcal I}$. 
\vspace{+3pt}

An argument used in \cite{CP15}, Corollary 8 (see eq.~(2.39) therein) allows us to obtain the next useful lemma. The proof, originally developed in \cite{CP15} is given in Appendix \ref{app:lemhitC} for the reader's convenience. 
\begin{lem}\label{lem:hitC}
For any $(x,y)\in \RR\times[0,1]$ we have
\begin{align}
\PP_{x,y}(\hat{\tau}_*=\check \tau)=\PP_{x,y}(\hat{\gamma}_*=\check \gamma)=1.
\end{align}
Equivalently for any $(z,y)\in\RR\times(0,1]$ we have
\begin{align}
\PP_{z,y}(\hat{\tau}_*=\check \tau)=\PP_{z,y}(\hat{\gamma}_*=\check \gamma)=1.
\end{align}
\end{lem}
The above lemma says that the process $(X,Y)$ (or equivalently $(Z,Y)$), upon hitting the optimal boundaries, will immediately enter the interior of the stopping set. 
This has the following important consequence
\begin{pro}\label{prop:convT}
Let $(x_n,y_n)_n$ be a sequence in $\CC$ and let $\hat{\tau}^n_{*}:=\hat{\tau}_*(x_n,y_n)$ and $\hat{\gamma}^n_*:=\hat{\gamma}_*(x_n,y_n)$ denote the corresponding hitting times for the process $(X^{x_n,y_n}, Y^{y_n})$. It follows that 
\begin{itemize}
\item[(i)] If $(x_n,y_n)\to (x_0,y_0)\in\CS_1$ as $n\to+\infty$, then $\hat{\tau}_*(x_n,y_n)\to 0$, $\PP$-a.s. 
\item[(ii)] If $(x_n,y_n)\to (x_0,y_0)\in\CS_2$ as $n\to+\infty$, then $\hat{\gamma}_*(x_n,y_n)\to 0$, $\PP$-a.s.
\end{itemize}
Notice that if $k>0$ the above holds with $(x_0,y_0)\neq(K,b_2^K)$.
\end{pro}
\begin{proof}
Let us consider (ii) and with no loss of generality let $x_0=b_2(y_0)$ (arguments as below apply also to $x_0=K$). Denote $\check{\gamma}^n:=\check\gamma(x_n,y_n)$. Since $\hat{\gamma}^n_*=\check\gamma^n$ by Lemma \ref{lem:hitC}, it is sufficient to prove that $\check\gamma^n\to 0$. 
In particular $\check\gamma (x_0,y_0)=0$, $\PP$-a.s.~by Lemma \ref{lem:hitC} and Lemma \ref{lem:reg-b}. Hence there exists  a set of null measure $\mathcal N$ such that $\check\gamma (x_0,y_0)=0$ and $(x,y) \rightarrow (X^{x,y},Y^y)$ is continuous, for all $\omega \in \Omega\setminus\mathcal N$.
Fix $\omega \in \Omega\setminus\mathcal N$ and an arbitrary $\alpha>0$. We can find $t<\alpha$ such that $X^{x_0,y_0}_t(\omega)< b_2(Y^{y_0}_t(\omega))$. It follows that for all $n$ sufficiently large $X^{x_n,y_n}_t(\omega)< b_2(Y^{y_n}_t(\omega))$  because $(X^{x_n,y_n}_t(\omega),Y^{y_n}_t(\omega))\to(X^{x_0,y_0}_t(\omega),Y^{y_0}_t(\omega))$ and $b_2$ is continuous. Therefore $\limsup_{n}\check \gamma^n(\omega)<\alpha$. Since $\alpha$ is arbitrary and the argument holds for a.e.~$\omega$ we obtain (ii). 

The proof of (i) follows from an analogous argument.    
\end{proof}

Now we can use the result above to obtain continuous differentiability of the value function. In preparation for that we need to recall some results concerning differentiability of the stochastic flow. In particular by \cite{Pr04}, Theorem 39, Chapter V.7 we can define the process
\begin{align}
\text{~for all $t\ge0$}, \;U^y_t:=\frac{\partial {Y}^y_t}{\partial y}\qquad\PP-\text{a.s.}
\end{align}
which is continuous in both $t$ and $y$ and solves the SDE
\begin{align}
dU^y_t=-\frac{\delta_0}{\sigma}(1-2{Y}^y_t)U^y_tdW_t,\qquad U^y_0=1\:\:\PP-\text{a.s.}
\end{align}
Notice that the couple $({Y},U)$ forms a Markov process and that $U^y_t$ is an exponential local martingale. Moreover, since the process $Y$ is bounded, it is not difficult to see that Novikov condition holds and $U^y_t$ is indeed an exponential martingale. Finally we also remark here that $(Y,U)$ is a strong solution of a SDE and notice that, using the explicit representation \eqref{def:Xhat}, we also have 
\[\frac{\partial}{\partial x}X^{x,y}_t= X^{1,y}_t\quad\PP-\text{a.s.}.\]
\vspace{+3pt}

For all $(x,y)\in\RR\times[0,1]$ we set 
\begin{align}
\label{def:u}&u(x,y):=V(x,y)-(x-K),\\[+4pt]
\end{align}
and define the process $(P_{t})_{t\ge0}$ as
\begin{align}
P_{t}=e^{r t}u(X_{t},Y_{t})+\int_0^{t}e^{-rs}(rK-\delta_0 X_sY_s)ds\qquad \PP_{x,y}-\text{a.s.}
\end{align}
Then from the semi-harmonic characterisation of the value function provided in Theorem \ref{value}, we obtain for any $T>0$
\begin{align}
\label{Pmart}&(P_{t\wedge\gamma_*\wedge\tau_*})_{t\le T}\quad\text{is a $\PP_{x,y}$ martingale}\\[+4pt]
\label{Psub}&(P_{t\wedge\tau_*})_{t\le T}\quad\text{is a $\PP_{x,y}$ sub-martingale}\\[+4pt]
\label{Psup}&(P_{t\wedge\gamma_*})_{t\le T}\quad\text{is a $\PP_{x,y}$ super-martingale}.
\end{align}
For future reference we also introduce
\begin{align}
\label{tauK}&\tau_K(x,y):=\inf\{t\ge 0\,:\,X^{x,y}_t\le K\}
\end{align}
and denote by $\overline \CC'$ the closure of $\CC'$.

\begin{pro}\label{prop:C1}
The value function $V$ is $C^1$ in $\RR_+\times(0,1)$ (possibly with the exception of the point $(K,b^K_2)$ if $k>0$). Moreover $v_{yy}$ (see \eqref{vH}) is continuous on $\overline \CC'$ (possibly~on $\overline \CC'\setminus (z_K,\overline y_K)$ if $k>0$).
\end{pro}
\begin{proof}
The value function is $C^1$ inside the continuation set $\CC$ by simply recalling that $v\in C^1$ in $\CC'$ (see the free boundary problem \eqref{freeb1}--\eqref{freeb3}). Therefore we only need to prove the $C^1$ property across the optimal boundaries. We provide full details for the continuity of $u_y:=\partial u/\partial y$ as the continuity of $u_x:=\partial u/\partial x$ follows analogous arguments up to trivial modifications.

Let us start by looking at points of $\partial \CS_1$, i.e.~the boundary of the stopping region for the buyer. Let us fix $(x_0,y_0)\in\partial \CS_1$ and let us pick $(x,y)$ inside the continuation set $\CC\cap\{x>K\}$. Later we will take limits $(x,y)\to(x_0,y_0)$ and use Proposition \ref{prop:convT}. 

Denote by $\tau_*=\tau_*(x,y)$ the first entry time of $(X^{x,y},Y^{y})$ into $\CS_1$ and by $\gamma_\eps=\gamma_*(x,y+\eps)$ the first entry time of $(X^{x,y+\eps},Y^{y+\eps})$ into $\CS_2$ for some $\eps>0$.
From $(i)$ of Lemma \ref{Lipschitz} and \eqref{def:u} we know that $u(x,y+\eps)-u(x,y)\le0$ since $V$ is non-increasing in $y$. In order to find a lower bound for $u(x,y+\eps)-u(x,y)$ we want to use the semi-harmonic property of $(P_{t})_{t\ge 0}$. For that we introduce the stopping time
$\lambda_\eps:=\tau_*\wedge\gamma_\eps\wedge\tau^\eps_K\wedge T$ where $T>0$ is fixed and $\tau^\eps_K=\tau_K(x,y+\eps)$. Notice that since $X^{x,y+\eps}\le X^{x,y}$ (see \eqref{def:Xhat}) then $\tau_K(x,y+\eps)\le\tau_K(x,y)$. Now, using \eqref{Psub} and \eqref{Psup} we obtain
\begin{align}\label{u1}
u&(x,y+\eps)-u(x,y)\\
\ge&\EE\Big[e^{-r\lambda_\eps}u(X^{x,y+\eps}_{\lambda_\eps},Y^{y+\eps}_{\lambda_\eps})+\int_0^{\lambda_\eps}e^{-rt}(rK-\delta_0 X^{x,y+\eps}_tY^{y+\eps}_t)dt\Big]\nonumber\\
&-\EE\Big[e^{-r\lambda_\eps}u(X^{x,y}_{\lambda_\eps},Y^{y}_{\lambda_\eps})+\int_0^{\lambda_\eps}e^{-rt}(rK-\delta_0 X^{x,y}_tY^{y}_t)dt\Big].\nonumber
\end{align}   
Notice that $0\le u\le \eps$ on $[K,+\infty)\times[0,1]$ and 
\begin{align*}
\tau_*\le \gamma_\eps\wedge\tau^\eps_K\wedge T\implies u(X^{x,y}_{\lambda_\eps},Y^y_{\lambda_\eps})=0\le u(X^{x,y+\eps}_{\lambda_\eps},Y^{y+\eps}_{\lambda_\eps})\\[+4pt]
\gamma_\eps\le\tau_*\wedge \tau^\eps_K\wedge T\implies u(X^{x,y+\eps}_{\lambda_\eps},Y^{y+\eps}_{\lambda_\eps})=\eps\ge u(X^{x,y}_{\lambda_\eps},Y^y_{\lambda_\eps})
\end{align*}
so that 
\begin{align}\label{lb}
\text{on}\:\: \{\tau_*\wedge\gamma_\eps\le \tau^\eps_K\wedge T\}\:\:\text{we have}\:\: u(X^{x,y+\eps}_{\lambda_\eps},Y^{y+\eps}_{\lambda_\eps})\ge u(X^{x,y}_{\lambda_\eps},Y^y_{\lambda_\eps}).
\end{align}
Using this fact in \eqref{u1} we get
\begin{align*}
u&(x,y+\eps)-u(x,y)\\[+4pt]
\ge&\EE\Big[\mathds{1}_{\{\tau^\eps_K\wedge T<\tau_*\wedge\gamma_\eps\}}e^{-r(\tau^\eps_K\wedge T)}\Big(u(X^{x,y+\eps}_{\tau^\eps_K\wedge T},Y^{y+\eps}_{\tau^\eps_K\wedge T})- u(X^{x,y}_{\tau^\eps_K\wedge T},Y^y_{\tau^\eps_K\wedge T})\Big)\Big]\\[+4pt]
&+\delta_0\EE\Big[\int_0^{\lambda_\eps}e^{-rt}(X^{x,y}_tY^y_t-X^{x,y+\eps}_tY^{y+\eps}_t)dt\Big].
\end{align*}
Now we use that $X^{x,y+\eps}\le X^{x,y}$ (see \eqref{def:Xhat}) and that $x\mapsto u(x,y)$ is non-increasing (as shown in the proof of $(iii)$ and $(iv)$ of Lemma \ref{monotonic}). Therefore from the right-hand side of the above inequality we easily get
\begin{align}\label{u2}
u&(x,y+\eps)-u(x,y)\\[+4pt]
\ge& \EE\Big[\mathds{1}_{\{\tau^\eps_K\wedge T<\tau_*\wedge\gamma_\eps\}}e^{-r(\tau^\eps_K\wedge T)}\Big(u(X^{x,y}_{\tau^\eps_K\wedge T},Y^{y+\eps}_{\tau^\eps_K\wedge T})- u(X^{x,y}_{\tau^\eps_K\wedge T},Y^y_{\tau^\eps_K\wedge T})\Big)\Big]\nonumber\\[+4pt]
&+\delta_0\EE\Big[\int_0^{\lambda_\eps}e^{-rt}X^{x,y}_t(Y^y_t-Y^{y+\eps}_t)dt\Big].\nonumber
\end{align}
Lower bounds can be provided for both terms on the right-hand side of the above expression. For the first term we recall $(iii)$ of Lemma \ref{Lipschitz} and get 
\begin{align}\label{u3}
\EE&\Big[\mathds{1}_{\{\tau^\eps_K\wedge T<\tau_*\wedge\gamma_\eps\}}e^{-r(\tau^\eps_K\wedge T)}\Big(u(X^{x,y}_{\tau^\eps_K\wedge T},Y^{y+\eps}_{\tau^\eps_K\wedge T})- u(X^{x,y}_{\tau^\eps_K\wedge T},Y^y_{\tau^\eps_K\wedge T})\Big)\Big]\nonumber\\[+4pt]
\ge&-C\,\EE\Big[\mathds{1}_{\{\tau^\eps_K\wedge T<\tau_*\wedge\gamma_\eps\}}e^{-r(\tau^\eps_K\wedge T)}(1+X^{x,y}_{\tau^\eps_K\wedge T})\left(Y^{y+\eps}_{\tau^\eps_K\wedge T}-Y^{y}_{\tau^\eps_K\wedge T}\right)\Big]\nonumber\\[+4pt]
=&-\eps\, C\,\EE\Big[\mathds{1}_{\{\tau^\eps_K\wedge T<\tau_*\wedge\gamma_\eps\}}e^{-r(\tau^\eps_K\wedge T)}(1+X^{x,y}_{\tau^\eps_K\wedge T})\Delta Y^\varepsilon_{\tau^\eps_K\wedge T}\Big]\nonumber\\[+4pt]
\ge &-\eps\, C\,\EE\Big[\mathds{1}_{\{\tau^\eps_K\wedge T<\tau_*\}}e^{-r(\tau^\eps_K\wedge T)}(1+X^{x,y}_{\tau^\eps_K\wedge T})\Delta Y^\varepsilon_{\tau^\eps_K\wedge T}\Big]
\end{align}
where $\Delta Y^\varepsilon_t:=\frac{1}{\varepsilon}(Y^{y+\eps}_{t}-Y^{y}_{t})$, and the final inequality follows by observing that $\{\tau_*\wedge\gamma_\eps>\tau^\eps_K\wedge T\}\subseteq\{\tau_*>\tau^\eps_K\wedge T\}$ and that the quantity under expectation is positive. 
For the integral term in \eqref{u2} we argue in a similar way and obtain
\begin{align}\label{u4}
\EE\Big[\int_0^{\lambda_\eps}e^{-rt}X^{x,y}_t(Y^y_t-Y^{y+\eps}_t)dt\Big]=&-\eps\,\EE\Big[\int_0^{\lambda_\eps}e^{-rt}X^{x,y}_t\Delta Y^\varepsilon_t dt\Big]\nonumber\\[+4pt]
\ge& -\eps\,\EE\Big[\int_0^{\tau_*\wedge \tau^\eps_K\wedge T}e^{-rt}X^{x,y}_t\Delta Y^\varepsilon_t dt\Big]
\end{align}

Collecting \eqref{u2}, \eqref{u3} and \eqref{u4} we find
\begin{align*}
&\frac{u(x,y+\eps)-u(x,y)}{\eps}\\
&\ge- C\,\EE\Big[\mathds{1}_{\{\tau^\eps_K\wedge T<\tau_*\}}e^{-r(\tau^\eps_K\wedge T)}(1+X^{x,y}_{\tau^\eps_K\wedge T})\Delta Y^\varepsilon_{\tau^\eps_K\wedge T}\Big]-\EE\Big[\int_0^{\tau_*\wedge \tau^\eps_K\wedge T}e^{-rt}X^{x,y}_t\Delta Y^\varepsilon_t dt\Big]
\end{align*}
and we now aim at taking limits as $\eps\to0$. In order to apply dominated convergence, it is sufficient to prove that the family of random variables $(X^{x,y}_\tau, \Delta Y^\varepsilon_\tau)$ is uniformly bounded in $L^4$ when $\tau$ ranges through all $[0,T]$ valued stopping times and $\varepsilon \in (0,1-y)$. Indeed, by Cauchy-Schwarz inequality, this will imply that $X^{x,y}_\tau\cdot \Delta Y^\varepsilon_\tau$ is bounded in $L^2$ uniformly with respect to $\eps$ and $\tau$.
\vspace{+3pt}

The bound for $X$ follows directly from the explicit expression \eqref{def:Xhat}. Note then that $\Delta Y^{\varepsilon}$ is an exponential martingale. Indeed, denoting $H^\varepsilon_t=-\frac{\delta_0}{\sigma}(1-Y^y_t-Y^{y+\varepsilon}_t)$, we have
\[ \Delta Y^\varepsilon_t= \exp\left[ \int_0^t H^\varepsilon_s dW_s - \frac{1}{2}\int_0^t(H^\varepsilon_s)^2 ds \right] .\] 
It follows that 
\begin{align*}
(\Delta Y^\varepsilon_t)^4 &= \exp\left[ 4\int_0^t H^\varepsilon_s dW_s - 4\frac{1}{2}\int_0^t (H^\varepsilon_s)^2 ds \right] \\
&=\exp\left[ 6\int_0^t (H^\varepsilon_s)^2 ds \right] \exp\left[ \int_0^t 4 H^\varepsilon
_s dW_s - \frac{1}{2}\int_0^t (4 H^\varepsilon_s)^2 ds \right]
\end{align*}
Since $H^\varepsilon$ is uniformly bounded by $\frac{\delta_0}{\sigma}$, the second term in the above expression is a martingale, and we deduce that for any stopping time $\tau$ taking values in $[0,T]$
\[ \EE[ (\Delta Y^\varepsilon)^4_\tau ] \leq \exp( 6T\frac{\delta_0^2}{\sigma^2}).\]
\vspace{+3pt}

\noindent Using that $\Delta Y^\varepsilon_t \rightarrow U^y_t$ almost surely for all $t\ge0$, as $\eps\to0$, we conclude 
\begin{align}\label{b1}
0\ge u_y(x,y)\ge &- C\,\EE\Big[\mathds{1}_{\{\tau_K\wedge T<\tau_*\}}e^{-r(\tau_K\wedge T)}(1+X^{x,y}_{\tau_K\wedge T})U^{y}_{\tau_K\wedge T}\Big]\\
&-\EE\Big[\int_0^{\tau_*\wedge \tau_K\wedge T}e^{-rt}X^{x,y}_tU^{y}_t dt\Big].\nonumber
\end{align}
In the above estimate we have used that 
\begin{align}
\tau^\eps_K\uparrow\tau_K \;\text{and} \;  \mathds{1}_{\{\tau^\eps_K\wedge T<\tau_*\}} \rightarrow \mathds{1}_{\{\tau_K\wedge T<\tau_*\}} \qquad\text{as $\eps\to 0$},
\end{align} 
which follows from the continuity of $(x,y)\rightarrow X^{x,y}$ and  the fact that $\PP(\tau_*=\tau_K)=0$ (see Proposition \ref{inclusionregions}).

Notice that the above estimates also imply that $U^y_\tau$ is bounded in $L^4$ and $X^{x,y}_\tau \cdot U^y_\tau$ is bounded in $L^2$, uniformly with respect to stopping times $\tau\in [0,T]$ and $(x,y) \in [K,x_0+1]\times (0,1)$. 

It remains to take limits as $(x,y)\to (x_0,y_0)$ with $(x,y)\in\CC$. 
By continuity of the sample paths $\tau_*(x,y)=\hat \tau_*(x,y)$ for $(x,y)\in\CC$. We use $(i)$ of Proposition \ref{prop:convT}, dominated convergence and \eqref{b1} (along with the fact that $\PP_{x_0,y_0}(\tau_K>0)=1$) to obtain
\begin{align}
\lim_{(x,y)\to(x_0,y_0)}u_y(x,y)=0.
\end{align}
The latter implies continuity of $u_y$ at $\partial\CS_1$.

To prove that $u_y$ is also continuous across $\partial \CS_2$ we need to argue in a slightly different way. Fix $(x_0,y_0)\in\partial \CS_2$ with $y_0<b^K_2$ and pick $(x,y)\in\CC$. With no loss of generality we consider $x_0=b_2(y_0)$ as the proof requires minor changes for $x_0=K$. We set $\gamma_*=\gamma_*(x,y)$ the first entry time of $(X^{x,y},Y^{y})$ into $\CS_2$ and denote by $\tau_\eps=\tau_*(x,y-\eps)$ the first entry time of $(X^{x,y-\eps},Y^{y-\eps})$ into $\CS_1$ for some $\eps>0$. Then we define $\eta_\eps:=\tau_\eps\wedge\gamma_*\wedge \tau_K\wedge T$ for some $T>0$. Again we recall that $\tau_K=\tau_K(x,y)\le\tau_K(x,y-\eps)$.

We know that $u(x,y)-u(x,y-\eps)\le0$ from $(i)$ of Lemma \ref{Lipschitz} and \eqref{def:u}. In order to find a lower bound we use \eqref{Psub} and \eqref{Psup} and get
\begin{align*}
u&(x,y)-u(x,y-\eps)\\
\ge&\EE\Big[e^{-r\eta_\eps}u(X^{x,y}_{\eta_\eps},Y^{y}_{\eta_\eps})+\int_0^{\eta_\eps}e^{-rt}(rK-\delta_0 X^{x,y}_tY^{y}_t)dt\Big]\nonumber\\
&-\EE\Big[e^{-r\eta_\eps}u(X^{x,y-\eps}_{\eta_\eps},Y^{y-\eps}_{\eta_\eps})+\int_0^{\eta_\eps}e^{-rt}(rK-\delta_0 X^{x,y-\eps}_tY^{y-\eps}_t)dt\Big].\nonumber
\end{align*}   
From this point onwards we can repeat the arguments used above up to trivial modifications. These allow us to conclude that $u_y$ is continuous across $\partial \CS_2$ with the possible exception of $(K,b^K_2)$, because Proposition \ref{prop:convT} does not hold at that point if $k>0$. 

As already mentioned, analogous arguments allow to prove that $u_x$ is also continuous everywhere with the possible exception of $(K,b^K_2)$.
It follows that $V\in C^1$ on $(\RR_+\times(0,1))\setminus(K,b^K_2)$ and $v\in C^1$ on $(\RR\times(0,1))\setminus (z_K,\overline y_K)$ (see \eqref{vH}). The latter and \eqref{freeb1} imply that $v_{yy}$ is continuous on $\overline\CC'\setminus (z_K,\overline y_K)$ as claimed.
\end{proof}

It remains to prove Lemma \ref{lem:reg-b} and for that it is convenient to change variables to the coordinate system $(z,y)$. We set
\begin{align}\label{def:w}
w(z,y)=u(F(z,y),y)
\end{align}
with the notation $w_z:=\partial w/\partial z$, $w_y:=\partial w/\partial y$ and $w_{yy}:=\partial^2 w/\partial y^2$. In these variables $\tau_K$ from \eqref{tauK} reads
\begin{align*}
\tau_K(z,y)=\inf\{t\ge 0\,:\,F(Z^z_t,Y^y_t)\le K\}.
\end{align*}
Notice that for $k > 0$ the boundary $c_1$ is non-decreasing and the stopping set $\CS_1'$ lies below it. Hence \eqref{regS1} is a consequence of standard arguments involving the law of iterated logarithm. Showing \eqref{regS2} for $k > 0$ is instead more difficult because $c_2$ is also non-decreasing but $\CS_2'$ lies above the boundary. A symmetric situation occurs for $k<0$. 

In what follows we first show that the classical smooth-fit condition holds and then prove that under our assumptions this implies Lemma \ref{lem:reg-b}. In the next lemma we only consider smooth-fit in those cases when the monotonicity of the boundary does not allow a direct proof of \eqref{regS1} or \eqref{regS2} based on the law of iterated logarithm.
\begin{lem}\label{lem:sm-f}
If $(z_0,y_0)\in \partial \CS'_1$ and $k<0$, then $w_y(z_0,y_0+)=0$. Analogously if $(z_0,y_0)\in \partial \CS'_2$, with $z_0<z_K$ and $y_0=c_2(z_0)$, and $k > 0$ then $w_y(z_0,y_0-)=0$. Finally, if $(z_0,y_0)\in \partial \CS'_2$, with $y_0=y_K(z_0)$ and $k < 0$ then $v_y(z_0,y_0+)=0$. 
\end{lem}
\begin{proof}
We carry out the proof under the assumption of $k> 0$ (see \eqref{def:k}). This induces no loss in generality as symmetric arguments hold for $k<0$.

Let $(z_0,y_0)\in \partial \CS'_2$ with $y_0=c_2(z_0)$.  Notice that for $y\in (y_0,y_K(z_0))$ we have $w_y(z_0,y)=0$. Also we know from the proof of $(i)$ in Lemma \ref{connectednessagain} that $w_y\ge0$ locally at $(z_0,y_0)$. We argue by contradiction and assume $w_y(z_0,y_0-)\ge\lambda_0>0$. The latter limit exists because $w_z$ is locally bounded (see \eqref{bWt}) and $|w_{yy}|\le c |w_z|$ in $\CC'$ due to \eqref{freeb1}, for a suitable $c>0$.

Fix $\eps>0$, consider the open rectangle $R_\eps:=(z_0,z_0+\eps)\times(y_0-\eps,y_0+\eps)$ and let $\rho_\eps=\inf\{t\ge0\,:\,(Z^{z_0}_t,Y^{y_0}_t)\notin R_\eps \}$. With no loss of generality we assume $\rho_\eps\le\tau_K\wedge\tau_*$ and from \eqref{Psub} we obtain
\begin{align}
w(z_0,y_0)\le \EE\left[e^{-r(t\wedge\rho_\eps)}w(Z^{z_0}_{t\wedge\rho_\eps},Y^{y_0}_{t\wedge\rho_\eps})+\int_0^{t\wedge\rho_\eps}e^{-r s}(rK-\delta_0 Y^{y_0}_sF(Z^{z_0}_s,Y^{y_0}_s))ds\right].
\end{align}
Since $w(\,\cdot\,,y)$ is non-increasing (see \eqref{bWt}) and $R_\eps$ is bounded we can find a constant $C_\eps>0$ depending on $R_\eps$ and such that 
\begin{align}
w(z_0,y_0)\le \EE\left[e^{-r(t\wedge\rho_\eps)}w(z_0,Y^{y_0}_{t\wedge\rho_\eps})+C_\eps (t\wedge\rho_\eps)\right].
\end{align}
Recalling that $w_{yy}(z_0,\cdot)$ is bounded on $[y_0-\eps,y_0+\eps]\setminus\{y_0\}$, we can apply It\^o-Tanaka formula to get
\begin{align}
w(z_0,y_0)\le & w(z_0,y_0)+\EE\left[\int_0^{t\wedge\rho_\eps}e^{-rs}\frac{\delta^2_0}{2\sigma^2}\left[Y^{y_0}_s(1-Y^{y_0}_s)\right]^2 w_{yy}(z_0,Y^{y_0}_s)\mathds{1}_{\{Y^{y_0}_s\neq y_0\}}ds\right]\nonumber\\
&+\EE\left[\frac{1}{2}\int_0^{t\wedge\rho_\eps}e^{-rs}(w_y(z_0,y_0+)-w_y(z_0,y_0-))dL^{y_0}_s(Y^{y_0})+C_\eps (t\wedge\rho_\eps)\right].
\end{align}
Boundedness of $w_{yy}(z_0,\,\cdot\,)$ and the assumption $w_y(z_0,y_0-)\ge\lambda_0$ give
\begin{align}\label{vil00}
0\le & -\frac{1}{2}\lambda_0\EE\left[\int_0^{t\wedge\rho_\eps}e^{-rs}dL^{y_0}_s(Y^{y_0})\right]+C'_\eps \EE\left[(t\wedge\rho_\eps)\right]
\end{align}
for some positive $C'_\eps>0$. For $0<p<1$, Burkholder-Davis-Gundy inequality and some algebra give
\begin{align}
\EE\left[\int_0^{t\wedge\rho_\eps}e^{-rs}dL^{y_0}_s(Y^{y_0})\right]\ge& \EE\left[e^{-rt}L^{y_0}_{t\wedge\rho_\eps}(Y^{y_0})\right]\nonumber\\
=&e^{-rt}\EE\left[|Y^{y_0}_{t\wedge\rho_\eps}-y_0|\right]\ge \frac{e^{-rt}}{(2\eps)^p}\EE\left[|Y^{y_0}_{t\wedge\rho_\eps}-y_0|^{1+p}\right]\nonumber\\
\ge& \frac{e^{-rt}}{(2\eps)^p} c_p\EE\left[\langle Y^{y_0}\rangle^{\frac{1+p}{2}}_{t\wedge\rho_\eps}\right]\ge\frac{e^{-rt}}{(2\eps)^p} c_{p,\eps}\EE\left[(t\wedge\rho_\eps)^{\frac{1+p}{2}}\right],
\end{align}
with $c_{p,\eps}>0$ depending on $p$ and $\eps$.
Plugging the latter inside \eqref{vil00} and letting $t\to 0$ we reach a contradiction. Therefore it must be $w_{y}(z_0,y_0-)=0$. 

The proof is entirely analogous for $(z_0,y_0)\in\partial\CS'_1$ and $k<0$. It is also worth noticing that for $(z_0,y_0)\in\partial\CS'_2$ with $y_0=y_K(z_0)$, the smooth-fit condition amounts to $v_y(z_0,y_0+)=0$ because the stopping payoff is $\eps_0$. Using $(iii)$ in Lemma \ref{connectednessagain} and arguments similar to those above we can prove that $v_y(z_0,y_0+)=0$ holds.
\end{proof}

\begin{proof}[Proof of Lemma \ref{lem:reg-b}]
Here we only consider the case $k > 0$ but the same results hold for $k<0$ and these can be proven by symmetric arguments. It is convenient to recall the function $w$ from \eqref{def:w}.

The proof of \eqref{regS1}, which we omit for brevity, is a straightforward consequence of the fact that $c_1$ is non-decreasing and $Y_t$ is non-degenerate away from $0$ and $1$, so that the law of iterated logarithm can be applied. The same rationale allows to prove that \eqref{regS2} holds for $y_0=y_K(z_0)$ for $z_0<z_K$.

To prove \eqref{regS2} with $y_0=c_2(z_0)$ and $z_0<z_K$ let us argue by contradiction and assume that $(x_0,y_0)\in\partial \CS_2\cap \{x>K\}$ is not regular or equivalently $(z_0,y_0)\in\partial \CS'_2\cap R_K$ is not regular (with $F(z_0,y_0)=x_0$), i.e.~\eqref{regS2} does not hold. Pick $y<y_0$ and $\eps>0$ such that $y+\eps<y_0$. Denote $\hat{\gamma}_\eps=\hat{\gamma}_*(z_0,y+\eps)$, $\hat{\tau}=\hat{\tau}_*(z_0,y)$, $\hat \gamma=\hat \gamma_*(z_0,y)$, $\tau^\eps_K=\tau_K(z_0,y+\eps)$. Notice that $\tau_K(z_0,y)\ge\tau_K(z_0,y+\eps)$, then from \eqref{Psub} and \eqref{Psup} and setting $\lambda_\eps:=\hat{\tau}\wedge\hat{\gamma}_\eps\wedge\tau^\eps_K\wedge T$ we obtain
\begin{align}\label{fatou1}
w(z_0,y+\eps)-w(z_0,y)\ge &\EE\left[\int_0^{\lambda_\eps}e^{-rt}\delta_0\left(Y^{y}_tF(Z^{z_0}_t,Y^{y}_t)-Y^{y+\eps}_tF(Z^{z_0}_t,Y^{y+\eps}_t)\right)dt\right]\nonumber\\
&+\EE\left[e^{-r\lambda_\eps}\left(w(Z^{z_0}_{\lambda_\eps},Y^{y+\eps}_{\lambda_\eps})-w(Z^{z_0}_{\lambda_\eps},Y^{y}_{\lambda_\eps})\right)\right]\nonumber\\
\ge &\EE\left[\int_0^{\lambda_\eps}e^{-rt}\delta_0\left(Y^{y}_tF(Z^{z_0}_t,Y^{y}_t)-Y^{y+\eps}_tF(Z^{z_0}_t,Y^{y+\eps}_t)\right)dt\right]
\end{align}
where in the last inequality we have used that $y\mapsto w(z,y)$ is non-decreasing as shown in the proof of Lemma \ref{connectednessagain}. 
Recall that $\frac{\partial}{\partial y}(yF(z,y))$ is strictly negative (see \eqref{h}) so that almost surely and for all $\eps>0$ we have
\begin{align*}
\int_0^{\lambda_\eps}e^{-rt}\delta_0\left(Y^{y}_tF(Z^{z_0}_t,Y^{y}_t)-Y^{y+\eps}_tF(Z^{z_0}_t,Y^{y+\eps}_t)\right)dt\ge 0.
\end{align*}

As in the proof of Proposition \ref{prop:C1} we have  $\tau^\eps_K\uparrow \tau_K$ as $\eps\to 0$. Moreover $\hat\gamma_\eps$ increases\footnote{Notice that, due to the geometry of $\CS'_2$, $(Z^{z_0},Y^{y_0})$ can only enter $\CS'_2$ by hitting $c_2$.} as $\eps\to 0$, hence $\hat\gamma_-:=\lim_{\eps\to 0}\hat \gamma_\eps\le \hat \gamma$, $\mathbb{P}$-a.s. To prove the reverse inequality we fix $\omega\in\Omega$ and pick $\delta>0$ such that $\hat\gamma(\omega)>\delta$. Then in particular we have 
\begin{align}\label{lsc1}
\inf_{0\le t\le \delta}(c_2(Z^{z_0}_t)-Y^{y}_t)(\omega)\ge c_\delta(\omega)>0
\end{align}
for some $c_\delta$. Recall that $(t,y)\mapsto U^{y}_t(\omega)$ is continuous, hence bounded on $[0,\delta]\times[0,1]$ by a constant $c'_\delta(\omega)>0$. Using that $|Y^{y+\eps}_t-Y^{y}_t|(\omega)\le c'_\delta(\omega)\cdot\eps$ we find  
\begin{align*}
\inf_{0\le t\le \delta}(c_2(Z^{z_0}_t)-Y^{y+\eps}_t)(\omega)\ge c_\delta(\omega)-c'_{\delta}(\omega)\cdot\eps 
\end{align*}
from \eqref{lsc1}. This implies that for all $\eps$ sufficiently small $\hat\gamma_\eps(\omega)> \delta$. Since $\delta$ was arbitrary we conclude $\lim_{\eps\to 0}\hat \gamma_\eps(\omega)= \hat \gamma(\omega)$. The argument holds for a.e.~$\omega\in\Omega$ hence we obtain 
\[
\lim_{\eps\to0}\hat \gamma_\eps=\hat\gamma,\qquad \mathbb{P}-\text{a.s.}
\]
Convergence of $\hat \gamma_\eps$ and $\tau^\eps_K$ imply 
\[
\lim_{\eps\to0}\lambda_\eps = \hat \tau\wedge\hat\gamma\wedge\tau_K\wedge T,\qquad\mathbb{P}-\text{a.s.}
\]

Dividing \eqref{fatou1} by $\eps$ and taking limits as $\eps\to 0$, we may use Fatou's theorem and the expression \eqref{h} for $\tfrac{\partial}{\partial y}(y F(z,y))$ to obtain 
\begin{align}\label{lims}
w_y(z_0,y)\ge -\delta_0\EE\left[\int_0^{T\wedge\hat{\tau}\wedge\hat{\gamma}\wedge\tau_K} e^{-rt}U^y_tF(Z^{z_0}_t,Y^{y}_t)\left(\frac{1-Y^y_t-\tfrac{\sigma^2}{\delta}}{1-Y^y_t}\right)dt\right].
\end{align}
Now we let $y\uparrow y_0$ and use that $\mathbb{P}$-a.s.~the following limits hold
\begin{align*}
&\hat{\gamma}_*(z_0,y)\downarrow\hat{\gamma}^+_*(z_0,y_0)\ge \hat{\gamma}_*(z_0,y_0),\\[+4pt]  
&\hat{\tau}_*(z_0,y)\uparrow \hat{\tau}_*(z_0,y_0) \quad\text{and}\quad \tau_K(z_0,y)\downarrow \tau_K(z_0,y_0).  
\end{align*}
In particular we notice that for the convergence of $\hat\tau_*$ we can use the same arguments as those used above for the convergence of $\hat\gamma_\eps$. Clearly 
\[\PP(\tau_K(z_0,y_0)>0)=\PP(\hat{\tau}_*(z_0,y_0)>0)=1, \]
and by assumption, $\PP(\hat{\gamma}_*(z_0,y_0)>0)>0$.
Using again Fatou's lemma, taking limits in \eqref{lims} the stopping time $\theta(z_0,y):=(\hat{\tau}\wedge\hat{\gamma}\wedge\tau_K)(z_0,y)$ converges to a stopping time $\theta(z_0,y_0)>0$, $\PP$-a.s. Hence $w_y(z_0,y_0-)>0$, which contradicts the smooth-fit principle proven in Lemma \ref{lem:sm-f}. In conclusion $(z_0,y_0)$ must be regular for $\CS'_2$, i.e.~\eqref{regS2} holds.
\end{proof}

\section{Existence of a Nash equilibrium}
Building on the results of the previous sections, we can prove the existence of a Nash equilibrium for our game with incomplete information. We recall here that the two main difficulties for such existence arise from the lack of uniform integrability of the stopping payoffs and the fact that the problem is bi-dimensional. In the rest of this section we make the next standing assumption.
\begin{ass}\label{ass:1}
We assume $\tfrac{\sigma^2}{\delta_0}> 1$.
\end{ass}

The next result will allow us to circumvent the lack of uniform integrability and it shows that the boundary $c_1$ of $\CS'_1$ is always strictly positive. 
\begin{lem}\label{boundedsupport}
For every $z \in \RR$ we have $c_1(z)>0$.
\end{lem}

\begin{proof}
Arguing by contradiction we assume that there exists $z_0\in\RR$ such that $c_1(z_0)=0$. Hence 
\begin{align}\label{contr}
(z_0,y)\notin\CS'_1~\text{and}~(F(z_0,y),y) \notin \CS_1\qquad\text{for all $y \in (0,1)$.} 
\end{align}
Since $F(\,\cdot\,,y)$ is increasing, properties of $\CS_1$ studied in Section \ref{sec:regions} imply that for fixed $h>0$ we may define a strip 
\begin{align*}
C_0(h):=\{(z,y)\in\RR\times(0,1)\,\tq\, F(z_0-h,y)\le x \le F(z_0,y)\}
\end{align*}
and $C_0(h)\cap \CS_1=\emptyset$.

In particular if we pick $y\in(0,1)$ and $x=F(z_0-h,y)$ then, assuming without loss of generality that $k > 0$ (see \eqref{def:k}), we have $\tau_* \ge h$, $\PP_{x,y}$-a.s. The latter follows by the fact that for all $t\in[0,h]$ the couple $(X^{x,y}_t,Y^{y}_t)$ lies in $C_0(h)$ because its joint distribution is supported along a curve $\{( F(z_0-h+kt,\zeta),\zeta),\, \zeta\in(0,1)\}$ (see Remark \ref{rem:supp}). Notice that for $k<0$ and with $x=F(z_0-h,y)$, monotonicity of $F(\,\cdot\,,\,y)$ and \eqref{contr} imply $\tau_* = +\infty$ $\PP_{x,y}$-a.s.).

Theorem \ref{value} gives
\begin{align}\label{ast}
V(x,y)&\le \EE_{x,y}\left[e^{-r(h\wedge \tau_*)}V(X_{h\wedge \tau_*},Y_{h\wedge \tau_*})\right]=\EE_{x,y}\left[e^{-r h}V(X_{h}, Y_{h})\right]\nonumber\\
&\le \EE_{x,y}[e^{-r h}X_{h}]=F(z_0-h,y)\EE_{y}[e^{-\delta_0 \int_0^h Y_t\,dt}M_h],
\end{align}
where $M_h=\exp(\sigma W_h -\frac{\sigma^2}{2}h)$. We aim at showing that for $y$ sufficiently close to zero we get
\begin{align*}
F(z_0-h,y)\EE_{y}\left[e^{-\delta_0 \int_0^h Y_t\,dt}M_h\right]\le G_1(F(z_0-h,y))=F(z_0-h,y)-K
\end{align*}
or equivalently
\begin{align}\label{theta1}
\Theta(y):=F(z_0-h,y)\left(1-\EE_{y}\left[e^{-\delta_0 \int_0^h Y_t\,dt}M_h\right]\right)\ge K.
\end{align}
The latter and \eqref{ast} lead to $V(x,y)\le x-K$, hence a contradiction.

Defining the probability measure $\PP^{(\sigma)}$ by 
$$
\frac{d\PP^{(\sigma)}_{y}}{d\PP_{y}}=M_h
$$ 
by Girsanov's theorem we have that $W_t^{(\sigma)}=W_t-\sigma t$, $t\ge0$ is a Brownian motion under $\PP^{(\sigma)}_{y}$. Moreover under the new measure $Y$ evolves according to 
\begin{align*}
Y_t= y-\frac{\delta_0}{\sigma}\int_0^t Y_s(1- Y_s)dW_s^{(\sigma)}-\delta_0\int_0^t Y_s(1- Y_s)ds.
\end{align*}
From the above dynamics it follows immediately that $\EE^{(\sigma)}_{y}( Y_t) \le y$ for all $t\ge 0$ and 
\begin{equation}\label{minorationexpectation}
\EE^{(\sigma)}_{y}(Y_t)\ge y- \delta_0 \int_0^t \EE^{(\sigma)}_{y}(Y_s)\, ds \ge y(1-\delta_0 t).
\end{equation}
Using the inequality $1-e^{-u} \ge u-\frac{u^2}{2}$ valid for $u \ge 0$, we have
\begin{align*}
\Theta(y)&=F(z_0-h,y)\EE_{y}\left[\left(1-e^{-\delta_0 \int_0^h Y_t\,dt}\right)M_h\right]\\
& \ge F(z_0-h,y)\left( \delta_0\EE^{(\sigma)}_{y}\left[\int_0^h Y_t\,dt\right]-\frac{\delta^2_0}{2}\EE^{(\sigma)}_{y}\left[\int_0^h Y_t\,dt\right]^2\right)\\
&\ge F(z_0-h,y)\left( y\delta_0\int_0^h(1-\delta_0 t)\,dt-\frac{\delta^2_0h}{2}\EE^{(\sigma)}_{y}\left[\int_0^h (Y_t)^2\,dt\right]\right),
\end{align*}
where for the last inequality we used \eqref{minorationexpectation} and Cauchy-Schwarz inequality. We aim at showing that
\begin{align}\label{**}
\EE^{(\sigma)}_{y}\left[\int_0^h (Y_t)^2\,dt\right] \le h y^2.
\end{align}
To see this, we observe that $Y^2$ is a supermartingale under the probability measure $\PP^{(\sigma)}$. Indeed, applying It\^o's formula we get
$$
d(Y^y_t)^2= -2\frac{\delta_0}{\sigma} (Y^y_t)^2(1-Y^y_t)dW^{(\sigma)}_t -\delta_0 (Y^y_t)^2(1-Y^y_t)dt +\frac{\delta_0^2}{\sigma^2}(Y^y_t)^2(1-Y^y_t)^2dt,
$$
and the drift part of the SDE is non-positive because $\frac{\sigma^2}{\delta_0} > 1$. Thus \eqref{**} holds as claimed.

Finally we obtain
\begin{align}\label{theta2}
\Theta(y)\ge\delta_0 h\, y\,F(z_0-h,y)\left( 1-\frac{\delta_0h}{2}(1+y)\right).
\end{align}
Recalling that $y\in[0,1]$, for $h$ sufficiently small we have $1> \tfrac{\delta_0h}{2}(1+C\,y)$. Moreover when $\sigma^2/\delta_0> 1$ it is immediate to check that $y\,F(z_0-h,y)\to+\infty$ as $y\to0$ (see \eqref{def:F}). In conclusion the right-hand side in \eqref{theta2} diverges, yielding the desired contradiction.
\end{proof}

We can now prove existence of a saddle point for our game.
\begin{pro}\label{Nash}
If $k>0$ the pair $(\gamma_*,\tau_*)$ defined in Theorem \ref{value} is a saddle point.
\end{pro}
\begin{proof}
Since Theorem \ref{value} guarantees the optimality of $\gamma_*$, i.e. 
\begin{align*}
V(x,y)\ge M_{x,y}(\tau,\gamma_*),\quad\text{for all $\tau\in\CT$,}
\end{align*}
it remains to prove the optimality of $\tau_*$, that is
\begin{align*}
V(x,y) \le M_{x,y}(\tau_*,\gamma),\quad\text{for all $\gamma\in\CT$,}
\end{align*}
Let $z\in\RR$ be fixed and set $x=F(z,y)$. Invoking Theorem \ref{value} and observing that for any fixed $t>0$ and $\gamma$,
\begin{align*}
&V( X_{ \tau_*}, Y_{ \tau_*})\indic_{ \{\tau_* \le t\wedge \gamma\}}=G_1(X_{\tau_*})\indic_{ \{\tau_* \le t\wedge \gamma \}}
\end{align*} 
we obtain 
\begin{align*}
V(x,y) \le& \EE_{x,y} \left[ e^{-r(t\wedge \tau_* \wedge \gamma)}V( X_{t \wedge \tau_* \wedge \gamma}, Y_{t\wedge \tau_* \wedge \gamma}) \right]\\
\le&\EE_{x,y} \left[  e^{-r \tau_*} G_1( X_{\tau_*}) \indic_{ \{\tau_* \le t\wedge \gamma \}} + e^{-r\gamma}G_2( X_{ \gamma})\indic_{ \{\gamma < t\wedge \tau_* \}} \right] \\
&+\EE_{x,y} \left[e^{-r t}V(X_{ t}, Y_{ t})\indic_{ \{t \le \tau_* \wedge \gamma \}} \right] 
\end{align*}
for any stopping time $\gamma$.

We now prove that the last term of the expression above converges to zero as $t\to +\infty$. Notice first that $c_1$ is non-decreasing (see Corollary \ref{cor:bb}) and therefore $\zeta\mapsto (b_1\circ c_1)(\zeta)$ is non-increasing due to Corollary \ref{cor:b}. 
For $t \le \tau_*$ we have $X_t^{x} = F(Z^z_t,Y^y_t) \le b_1(Y_t)$, which implies that $Y^y_t\ge c_1(Z^z_t)\ge c_1(z)$ after the change of variables. Then $b_1(Y_t)\le (b_1\circ c_1)(z)$ and we have the uniform bound $X_t^{x}\le (b_1\circ c_1)(z)=:a_z$ for $t \le \tau_*$. Notice that $c_1(z)>0$ thanks to Lemma \ref{boundedsupport}, so that we also have $a_z<+\infty$.

Using such bound we get
\begin{align*}
\EE_{x,y} \left[e^{-r t}V( X_{t}, Y_{t})\indic_{ \{t \le \tau_* \wedge \gamma \}} \right] \le G_2\left(a_z\right)e^{-r t}\to 0,\quad\text{as $t\to0$.}
\end{align*}
Next, the monotone convergence theorem yields
\begin{align*}
V(x,y)\le&  \lim_{t \to +\infty}\EE_{x,y} \left[  e^{-r \tau_*} G_1( X_{\tau_*}) \indic_{ \{\tau_* \le t\wedge \gamma \}} + e^{-r\gamma}G_2( X_{ \gamma})\indic_{ \{\gamma < t\wedge \tau_* \}} \right] \\
=&\EE_{x,y} \left[  e^{-r \tau_*} G_1(X_{\tau_*}) \indic_{ \{\tau_* \le \gamma \}} + e^{-r\gamma_*}G_2( X_{ \gamma})\indic_{ \{\gamma < \tau_* \}} \right] \\
=& M_{x,y}(\tau_*,\gamma),
\end{align*}
that is, $\tau_*$ is optimal for the buyer. 
\end{proof}

Let us now analyze the case $k<0$, for which we prove existence of a Nash equilibrium under stronger assumptions on the parameters. We start with an auxiliary lemma, which will require the following assumption (recall also that $\sigma^2/\delta_0>1$ by Assumption \ref{ass:1}).
\begin{ass}\label{ass:strong}
We take $r$ such that
\begin{align}\label{bound-r}
\frac{\delta_0}{\sigma^2}\left(\frac{\delta_0+\sigma^2}{2}\right)<r<\left(\frac{\delta_0+\sigma^2}{2}\right).
\end{align}
\end{ass}
Notice that \eqref{bound-r} indeed implies $k<0$.
\begin{lem}\label{case2}
Under Assumption \ref{ass:strong} it holds that:
\[
\lim_{t\to\infty} e^{-rt}F(z+kt,c_1(z+kt)) = 0,\qquad \forall z\in \RR.
\]
\end{lem}
\begin{proof}
First note that
\[
e^{-rt}F(z+kt,c_1(z+kt))=e^{z+(k-r)t}\left(\frac{1-c_1(z+kt)}{c_1(z+kt)}\right)^{\frac{\sigma^2}{\delta_0}}.
\]
Then recall that $c_1(\cdot)\le y_K(\cdot)$ (see \eqref{RK}) and since $k<0$ then $c_1(z+kt)\to0$ as $t\to+\infty$. 
It is therefore sufficient to prove that as $t\to\infty$  
\begin{align}\label{req} 
\frac{1}{c_1(z+kt)} \leq c\,e^{\alpha t},
\end{align}
for some constants $c>0$ and $\alpha < \frac{\delta_0}{\sigma^2}(r-k)$.

Define $\lambda_{a}=\inf\{ t>0 \,|\, Y_t \leq a\}$. Let $z\in \RR$, $y>c_1(z)$ and $x=F(z,y)$. Note that since $z\rightarrow c_1(z)$ is non-decreasing and $k<0$, we have $\tau_* \geq \lambda_{c_1(z)}$ $\PP_{x,y}$-almost surely. Therefore, for all $t\geq 0$ Theorem \ref{value} gives
\begin{align}\label{eq:n1}
V(x,y)&\le \EE_{x,y}\left[e^{-r(t\wedge \lambda_{c_1(z)})}V\left( X_{t\wedge \lambda_{c_1(z)}}, Y_{t\wedge \lambda_{c_1(z)}}\right)\right]\\
&=\EE_{x,y}\left[e^{-r t}V\left( X_{t}, Y_{t}\right)\indic_{\{t<\lambda_{c_1(z)}\}} +e^{-r \lambda_{c_1(z)}} V\left( X_{\lambda_{c_1(z)}},c_1(z)\right)\indic_{\{\lambda_{c_1(z)}\leq t\}}\right].\nonumber
\end{align}

On the event $\{t<\lambda_{c_1(z)}\}$ we have $X_t\le b_1(Y_t)$ and $Y_t\ge c_1(z)$, $\PP_{x,y}$-a.s., so that $X_t\le (b_1\circ c_1)(z)=:a_z<+\infty$, $\PP_{x,y}$-a.s.~(as in the proof of Proposition \ref{Nash}).
The latter implies  
\[ 
e^{-r t}V( X_{t}, Y_{t})\indic_{\{t<\lambda_{c_1(z)}\}} \leq e^{-r t} X_{t}\indic_{\{t<\lambda_{c_1(z)}\}} \leq e^{-rt}a_z\indic_{\{t<\lambda_{c_1(z)}\}}
\]
and hence 
\[
\lim_{t\to 0}\EE_{x,y}\left[e^{-r t}V( X_{t}, Y_{t})\indic_{\{t<\lambda_{c_1(z)}\}}\right]= 0,\quad\text{$\PP_{x,y}$-a.s.}
\]
Taking limits in \eqref{eq:n1} as $t\to\infty$ and using monotone convergence, we deduce that
\begin{align*} 
V(x,y) \leq&\, \EE_{x,y}\left[e^{-r \lambda_{c_1(z)}} V\left( X_{\lambda_{c_1(z)}},c_1(z)\right)\indic_{\{\lambda_{c_1(z)}<\infty\}}\right]\\
\leq&\, \EE_{x,y}\left[e^{-r \lambda_{c_1(z)}} X_{\lambda_{c_1(z)}}\indic_{\{\lambda_{c_1(z)}<\infty\}}\right] \\
 =&\,\EE_y\left[e^{-r \lambda_{c_1(z)}} F(z+k \lambda_{c_1(z)}, c_1(z))\indic_{\{\lambda_{c_1(z)}<\infty\}}\right] \\
 =&\, e^z\left( \frac{1-c_1(z)}{c_1(z)}\right)^{\frac{\sigma^2}{\delta_0}}\EE_{y}\left[e^{(-r+k) \lambda_{c_1(z)}}\indic_{\{\lambda_{c_1(z)}<\infty\}}\right]
\end{align*}

In order to compute the Laplace trasform of $\lambda_{c_1(z)}$ we need to recall the fundamental solutions of $\CL_{Y} f- (r-k)f =0$, where $\CL_Y$ denotes the infinitesimal generator of the diffusion $Y$. Letting $\psi$ be the unique positive, increasing solution and $\phi$ the unique positive, decreasing one, we have 
\[ 
\psi(y)=y^{\beta}(1-y)^{1-\beta} \quad\text{and}\quad \phi(y)=y^{1-\beta}(1-y)^{\beta} 
\]
where $\beta=\frac{\sigma^2}{\delta_0}+1$ is the largest solution of 
\[ 
\beta(\beta-1)=\frac{2\sigma^2(r-k)}{\delta_0^2}= \frac{\sigma^2}{\delta_0}\left(\frac{\sigma^2}{\delta_0}+1\right).
\]
In terms of $\phi$ the Laplace transform of $\lambda_{c_1(z)}$ reads (recall that $y>c_1(z)$)
\[
\EE_y\left[e^{(-r+k) \lambda_{c_1(z)}}\indic_{\{\lambda_{c_1(z)}<\infty\}}\right] = \frac{y^{1-\beta}(1-y)^\beta}{c_1(z)^{1-\beta}(1-c_1(z))^\beta}.
\]
In conclusion, for any $z\in\RR$ and $y>c_1(z)$, taking $x=F(z,y)$ we have (recall \eqref{vH})
\begin{align}\label{bb1} 
v(z,y)=&\,V(x,y)\\
\leq &\, e^z\left( \frac{1-c_1(z)}{c_1(z)}\right)^{\frac{\sigma^2}{\delta_0}} \frac{y^{1-\beta}(1-y)^\beta}{c_1(z)^{1-\beta}(1-c_1(z))^\beta}= \frac{(1-y)}{(1-c_1(z))}F(z,y).\nonumber
\end{align}

Now we fix $z\in\RR$ and pick $a>1$ such that $a\,c_1(z)<1$. Since $v(z+kt,a\,c_1(z+kt))\ge F(z+kt,a\,c_1(z+kt))-K$ for all $t\ge 0$ we can use the latter and \eqref{bb1}, replacing $(z,y)$ therein by $\big(z+kt,a\,c_1(z+kt)\big)$, to estimate  
\begin{align*} 
-K \leq&\, (v-F)(z+kt,a\,c_1(z+kt)) \leq -\frac{(a-1)c_1(z+kt)}{1-c_1(z+kt)}F(z+kt,a\,c_1(z+kt)).
\end{align*}
Simple algebra gives
\[
\frac{1}{c_1(z+kt)}  \leq c\, e^{\alpha t}
\]
for some constant $c>0$ depending on $z$, $K$ and $a$, and with $\alpha=-k \delta_0/(\sigma^2-\delta_0)$. Now Assumption \ref{ass:strong} implies that $\frac{-k \delta_0}{\sigma^2-\delta_0}<\frac{\delta_0}{\sigma^2}(r-k)$ as required in \eqref{req}.
\end{proof}

\begin{pro}\label{Nash2}
Under Assumption \ref{ass:strong} the pair $(\tau_*,\gamma_*)$ is a saddle point.
\end{pro}
\begin{proof}
As in Proposition \ref{Nash}, we only have to prove the optimality of $\tau_*$ and we argue in a similar way.
Let $z\in\RR$ be fixed and set $x=F(z,y)$, then as in the proof of Proposition \ref{Nash} we find 
\begin{align}\label{fin}
V(x,y) \le&\, \EE_{x,y} \left[ e^{-r(t\wedge \tau_* \wedge \gamma)}V(X_{t \wedge \tau_* \wedge \gamma}, Y_{t\wedge \tau_* \wedge \gamma}) \right]\nonumber\\
=&\,\EE_{x,y} \left[  e^{-r \tau_*} G_1( X_{\tau_*}) \indic_{ \{\tau_* \le t\wedge \gamma \}} + e^{-r\gamma}G_2( X_{ \gamma}, Y_{ \gamma})\indic_{ \{\gamma < t\wedge \tau_* \}} \right] \\
&+\EE_{x,y} \left[e^{-r t}V( X_{ t}, Y_{ t})\indic_{ \{t\le \tau_*\wedge \gamma \}} \right] \nonumber
\end{align}
for any stopping time $\gamma$ and any $t$. Under $\PP_{x,y}$ we have $X_t = F(z+kt,Y_t)$ and for $t < \tau_*$ we have $Y_t\ge c_1(z+kt)$, which implies $X_t \le F(z+kt,c_1(z+kt))$. The latter gives 
\[
\EE_{x,y}\left[e^{-r t}V( X_{t}, Y_{t})\indic_{ \{t \le \tau_*\wedge \gamma \}} \right] \le e^{-r t}F(z+kt,c_1(z+kt)),
\]
which goes to zero according to Lemma \ref{case2}. Then taking limits as $t\to \infty$ in \eqref{fin} and using also monotone convergence we conclude the proof.
\end{proof}

\section{Concluding remarks}\label{sec:concl}

Our approach relies on the ability to obtain a two-dimensional Markovian dynamic for the stock price and the expected value of the dividend rate (given the observations of the stock). This stems from fact that the dividend rate has a two-point distribution. Similarly, if $D$ has a discrete distribution taking $n$ values, we can use filtering methods to reduce the problem to a stopping game on a $n$-dimensional degenerate diffusion. However, it should be clear at this point that a free boundary analysis of such problem is likely to be extremely convoluted. An even more complex situation arises when the dividend rate is allowed to take infinitely many values. In that case the dynamics obtained via our filtering approach may easily lead to a formulation of the game which is intractable with free boundary methods. 

An alternative approach relies on the use of a Girsanov transformation. While it falls outside the scope of the present paper to perform a fully rigorous analysis of this method, we believe it may be useful for future research to outline the main ideas of this approach and point out some questions that arise naturally. 

Letting a process $(\beta_t)_{t\ge 0}$ be defined by
\[\beta_t:=B_t+\sigma^{-1}(r-\delta_0D) t,\qquad t\ge0, \]
we have that the stock price in \eqref{defiX} reads
\begin{align}\label{Sb}
S_t=S_0\exp\left(\sigma\beta_t-\tfrac{\sigma^2}{2}t\right).
\end{align}
Moreover, the process $\beta$ is an $\CF$-Brownian motion for $t\in[0,T]$, under the measure $\mathbb{Q}$ defined by
\begin{align}\label{eq:Q}
N_T=\frac{d \mathbb{Q}}{d\mathbb{P}}\bigg|_{\mathcal F_T}:=\exp\left(\tfrac{\delta_0 D-r}{\sigma}\beta_T+\tfrac{(\delta_0D-r)^2}{2\sigma^2}T\right),
\end{align}  
for all $T \geq 0$, where $\CF$ is the (augmented) filtration generated by $B$ and $D$.

Although it is not possible, in general, to perform the change of measure on $\mathcal{F}_\infty$ (see, e.g., \cite[pp.~192-193]{KS}), let us set this problem aside for now and assume that the distribution of $D$ is sufficiently `nice' to allow the use of \eqref{eq:Q} in order to rewrite \eqref{gameformulationone} as
\begin{align*}
M(\tau,\gamma)=\mathbb{E}^{\mathbb{Q}}\left[N_{\tau\wedge\gamma}\left(e^{-r\tau} G_1(S_\tau)\mathds{1}_{\{\tau\le \gamma\}}+e^{-r\gamma} G_2(S_\gamma)\mathds{1}_{\{\tau>\gamma\}}\right)\right].
\end{align*}
Now we define a process $(L_t)_{t\ge 0}$ with $L_t:= \mathbb{E}^{\mathbb{Q}}[N_t|\mathcal{F}^S_t]$. Thanks to \eqref{eq:Q}, using the fact that $D$ and $\beta$ are independent under $\mathbb{Q}$, and expressing $\beta$ in terms of $S$ (see \eqref{Sb}), we have $L_t=f_D(t,S_t,S_0)$ for some function $f_D$, depending on the specific distribution of $D$. Then, the game's payoff reads
\begin{align}\label{eq:Gb}
M(\tau,\gamma)=\mathbb{E}^{\mathbb{Q}}\left[f_D(\tau\wedge\gamma,S_{\tau\wedge\gamma},S_0)\left(e^{-r\tau} G_1(S_\tau)\mathds{1}_{\{\tau\le \gamma\}}+e^{-r\gamma} G_2(S_\gamma)\mathds{1}_{\{\tau>\gamma\}}\right)\right]
\end{align}
where we note that $f_D$ can be computed explicitly in some cases\footnote{For example, in the simple case of $D\sim N(0,1)$ we have 
\[
f_D(t,s,s_0)=(1+t(\delta_0/\sigma)^2)^{-1/2}\exp\left[g(t,s,s_0)/(1+t(\delta_0/\sigma)^2)\right],
\] 
with $g(t,s,s_0):=\tfrac{1}{2t}q^2(t,s,s_0)-\tfrac{t}{2}[\tfrac{r}{\sigma}-\tfrac{1}{t}q(t,s,s_0)]^2/(1+t(\delta_0/\sigma)^2)$ and $q(t,s,s_0):=\sigma^{-1}\ln(s/s_0)+\sigma t/2$.}.

The construction above holds for any law of $D$ that allows to justify the change of measure on $\mathcal{F}_\infty$ (a seemingly non-trivial task). However, under the expectation, the resulting game's payoff depends explicitly on the initial value of the stock price $S_0$. One way to circumvent this issue would be to consider $S_0$ as a `parameter' in the game formulation \eqref{eq:Gb}, and treat it independently of the initial value of the process $S$. That is, we would fix an arbitrary $\bar s_0$ and study the game with payoff
\begin{align}\label{eq:Gb2}
\mathbb{E}^{\mathbb{Q}}\left[f_D(\tau\wedge\gamma,S_{\tau\wedge\gamma},\bar s_0)\left(e^{-r\tau} G_1(S_\tau)\mathds{1}_{\{\tau\le \gamma\}}+e^{-r\gamma} G_2(S_\gamma)\mathds{1}_{\{\tau>\gamma\}}\right)\right],
\end{align}
where the process $S$ starts from an arbitrary point $S_0$, possibly different from $\bar s_0$.
Now, for each $\bar s_0$ one must solve the Dynkin game with payoff as in \eqref{eq:Gb2}, which remains a challenging task due to the (generally) convoluted expression of $f_D$. Moreover, the shapes of the continuation and stopping region need to be studied not only as functions of time but also as functions of the parameter $\bar s_0$. 

It is interesting to notice that the approach outlined above corresponds to the study of {\em pre-commitment strategies} in the closely related literature on time-inconsistent control/stopping problems. (The interested reader may consult, e.g., \cite{CL19} and references therein, for a recent detailed study on a class of time-inconsistent stopping problems where time inconsistency stems from a `parametric' dependence of the gain function on the starting point of the process, i.e., the analogue of our $f_D(\cdot,\cdot,S_0)$). To the best of our knowledge, time-inconsistent Dynkin games have never been addressed in the literature. Moreover, there seems to be no clear consensus, as to whether the pre-commitment strategy is conceptually the best way forward in time-inconsistent stochastic optimisation problems. This interesting question is left for future research.

\appendix

\section{Proof of Theorem \ref{value}}
The main idea of the proof is to approximate our game by a sequence of games with bounded stopping payoffs indexed by $n\in\mathbb{N}$. For each approximating problem we can apply the results of \cite{ekstrompeskir} regarding existence of the value and of a saddle point. Eventually we pass to the limit as $n\to\infty$ to obtain the existence of the value for the game with unbounded payoffs.

For $n\geq 1$ let us define the functions $G_i^{(n)}(x)=G_i(x\wedge n)$, $i=1,2$. Next for $\tau,\gamma \in \CT$ let us introduce the the associated payoff
\begin{align}\label{eq:Mn} 
M^{(n)}_{x,y}(\tau,\gamma)=\EE_{x,y}[e^{-r \tau}G_1^{(n)}(X_\tau)\indic_{\{\tau \leq \gamma\}} + e^{-r \gamma}G_2^{(n)}(X_{\gamma})\indic_{\{\gamma < \tau\}} ].
\end{align}
According to Theorem 2.1. in \cite{ekstrompeskir}, the game with payoff \eqref{eq:Mn} has a value, i.e. 
\begin{align*}
V^{(n)}(x,y)= \sup_\tau \inf_\gamma M^{(n)}_{x,y}(\tau, \gamma) =\inf_\gamma \sup_\tau M^{(n)}_{x,y}(\tau, \gamma).
\end{align*}
Moreover, the stopping times 
\[ \tau_n = \inf\{ t \geq 0 \,|\, V^{(n)}(X_t,Y_t)=G_1^{(n)}(X_t) \},\]
\[ \gamma_n = \inf\{ t \geq 0 \,|\, V^{(n)}(X_t,Y_t)=G_2^{(n)}(X_t) \},\]
form a Nash equilibrium. Since 
\[G_1^{(n)}(x) \leq  V^{(n)}(x,y) \leq \sup_{\tau} M^{(n)}_{x,y}(\tau,+\infty) \leq  (n-K)^+\]
and $G_1^{(n)}(x)= (n-K)^+$ for $x\geq n$, then $V^{(n)}(x,y)=G^{(n)}_1(x)$ for $(x,y) \in [n,+\infty)\times [0,1]$.
The latter implies that 
\begin{align*}
\{(x,y)\,\tq\,V^{(n)}(x,y)=G_2^{(n)}(x)\} \subset [0,n]\times[0,1]  =\{ (x,y)\,\tq\, G_2(x)=G_2^{(n)}(x)\}
\end{align*}
 and therefore 
\[ \gamma_n = \inf\{ t \geq 0 \,|\, V^{(n)}(X_t,Y_t)=G_2(X_t) \}.\]

Concerning the value of the approximating game, it is easy to check that Lemma \ref{Lipschitz} holds for $V^{(n)}$ with the same proof. Moreover the sequence $M^{(n)}_{x,y}(\tau,\gamma)$ is non-decreasing in $n$ and it is bounded from above by $M_{x,y}(\tau,\gamma)$. Hence the sequence $(V^{(n)})_{n \geq 1}$ is non-decreasing in $n$ with 
\begin{align}\label{eq:Vinf}
V^{(n)}(x,y) \leq \underline V(x,y) \leq \overline V(x,y)\quad\text{and} \quad V^\infty(x,y):= \lim_{n\to\infty} V^{(n)}(x,y)
\end{align}
for all $(x,y)\in\mathbb R_+\times[0,1]$. In particular $V^{(n)} \leq \overline V$ implies
\[ \gamma_n \geq \gamma_*:=\inf\{ t \geq 0 \,|\, \overline V(X_t,Y_t)=G_2(X_t) \}\quad\text{for all $n\ge 1$}.\]
Since $V^{(n)}$ is non-decreasing in $n$ then $\gamma_n$ is non-increasing and we set $\gamma_\infty := \lim_{n\to\infty} \gamma_n$.
\vspace{+5pt}

Now we aim at showing that $V^\infty\ge \overline V$ so that \eqref{eq:Vinf} implies $\underline V=\overline V$ and therefore the value exists and it coincides with $V^\infty$. 
For all $\tau \in \CT$, we have
\begin{align*}
M^{(n)}_{x,y}(\tau,\gamma_n)  =& \EE_{x,y}\left[e^{-r \tau}G_1^{(n)}(X_\tau)\indic_{\{\tau \leq \gamma_n\}} + e^{-r \gamma_n}G_2^{(n)}(X_{\gamma_n})\indic_{\{\gamma_n < \tau\}} \right] \\[+2pt]
=& \EE_{x,y}\left[e^{-r \tau}G_1(X_\tau)\indic_{\{\tau \leq \gamma_n\}} + e^{-r \gamma_n}G_2(X_{\gamma_n})\indic_{\{\gamma_n < \tau\}}\right]\\[+2pt]
&-\EE_{x,y}\left[e^{-r \tau}(X_\tau-n)\indic_{\{X_\tau \geq n\}}\indic_{\{\tau \leq \gamma_n\}}\right].
\end{align*}

Observe that
\begin{align}\label{eq:bound}
0 \le \EE_{x,y}[e^{-r \tau}(X_\tau-n)\indic_{\{X_\tau \geq n\}}\indic_{\{\tau \leq \gamma_n\}} ] \le \EE_{x,y}[e^{-r \tau}X_\tau\indic_{\{X_\tau \geq n\}} ] 
\end{align}
and recall that $\EE_{x,y}[e^{-r \tau}X_\tau ]\le x$ by \eqref{op-sam}.
Using dominated convergence in \eqref{eq:bound} we obtain 
\[
\lim_{n\to\infty}\EE_{x,y}[e^{-r \tau}(X_\tau-n)\indic_{\{X_\tau \geq n\}}\indic_{\{\tau \leq \gamma_n\}} ]=0.
\]

On the other hand, Fatou's Lemma implies 
\begin{align*}
\liminf_{n\to\infty}&\, \EE_{x,y}\left[e^{-r \tau}G_1(X_\tau)\indic_{\{\tau \leq \gamma_n\}} + e^{-r \gamma_n}G_2(X_{\gamma_n})\indic_{\{\gamma_n < \tau\}} \right]\\[+2pt]
 \geq& \EE_{x,y}\left[e^{-r \tau}G_1(X_\tau)\indic_{\{\tau \leq \gamma_\infty\}} + e^{-r \gamma_\infty}G_2(X_{\gamma_\infty})\indic_{\{\gamma_\infty < \tau\}} \right]=M_{x,y}(\tau,\gamma_\infty).
\end{align*}
Collecting the above limits we deduce that
\begin{align}\label{eq:limM} 
\liminf_{n\to\infty} M^{(n)}_{x,y}(\tau,\gamma_n) \geq M_{x,y}(\tau,\gamma_\infty).
\end{align}

Now, for $\varepsilon>0$, let $\tau_\varepsilon$ be such that
\[ M_{x,y}(\tau_\varepsilon,\gamma_\infty) \geq \sup_\tau  M_{x,y}(\tau,\gamma_\infty)-\varepsilon.\]
Using optimality of $\gamma_n$ in the approximating problem, and \eqref{eq:limM} we obtain 
\begin{align*}
V^\infty(x,y) &= \lim_{n\to\infty} V^{(n)}(x,y) = \lim_{n\to\infty} \sup_\tau M^{(n)}_{x,y}(\tau,\gamma_n) \geq \liminf_{n\to\infty} M^{(n)}_{x,y}(\tau_\varepsilon,\gamma_n) \\ 
& \geq M_{x,y}(\tau_\varepsilon,\gamma_\infty) \geq \sup_\tau  M_{x,y}(\tau,\gamma_\infty)-\varepsilon  \geq \overline V (x,y) - \varepsilon.
\end{align*}
Finally, letting $\varepsilon\to 0$ and recalling \eqref{eq:Vinf}, we obtain
\[ \underline V(x,y) \geq   V^\infty(x,y) \geq \overline V (x,y),\]
and hence the existence of the value $V:=V^\infty$. As a byproduct we also obtain that $\gamma_\infty$ is optimal for player 2, that is
\[
V(x,y)=\sup_\tau  M_{x,y}(\tau,\gamma_\infty).
\]
\vspace{+5pt}

Next we want to prove optimality of $\gamma_*$ and super/sub-martingale properties of $V$. For all $n$ and any $\tau\in\CT$ we have (see \cite[Thm.~2.1.]{ekstrompeskir}) 
\begin{align*}
 V^{(n)}(x,y) \geq& \EE_{x,y}[ e^{-r (\tau \wedge \gamma_n)}V^{(n)}(X_{\tau \wedge \gamma_n},Y_{\tau \wedge \gamma_n})]\\
 =&\EE_{x,y}[ e^{-r \tau }V^{(n)}(X_{\tau},Y_{\tau })\indic_{\{\tau \leq \gamma_n\}}]+\EE_{x,y}[ e^{-r \gamma_n }G_2(X_{\gamma_n})\indic_{\{\tau > \gamma_n\}}]\\
 \ge& \EE_{x,y}[ e^{-r \tau }V^{(n)}(X_{\tau},Y_{\tau })\indic_{\{\tau \leq \gamma_n\}}]+\EE_{x,y}[ e^{-r \gamma_n }V(X_{\gamma_n},Y_{\gamma_n})\indic_{\{\tau > \gamma_n\}}]\\
 =&\EE_{x,y}[ e^{-r (\tau \wedge \gamma_n)}V(X_{\tau \wedge \gamma_n},Y_{\tau \wedge \gamma_n})] \\
&+ \EE_{x,y}[ e^{-r \tau }(V^{(n)}(X_{\tau},Y_{\tau })-V(X_{\tau},Y_{\tau }))\indic_{\{\tau \leq \gamma_n\}}]
\end{align*}
where in the second inequality we used that $G_2\ge V$. Now we take limits as $n\to\infty$. Recalling that $V^{(n)}\le V$, that $0\le V(x,y) \le x+\eps_0$ and that $e^{-r\tau}X_{\tau}$ is integrable, the second term in the last expression above converges to zero by dominated convergence. Moreover, Fatou's Lemma yields,
\begin{equation}\label{supermart_gammainfty}
V(x,y) \ge \EE_{x,y}[ e^{-r (\tau \wedge \gamma_\infty)}V(X_{\tau \wedge \gamma_\infty},Y_{\tau \wedge \gamma_\infty})].
\end{equation}

Since $\tau\in\CT$ was arbitrary the process $e^{-r (t \wedge \gamma_\infty)}V(X_{t \wedge \gamma_\infty},Y_{t \wedge \gamma_\infty})$, $t\ge0$ is a super-martingale.
Noticing that $\gamma_\infty \ge \gamma_*$ and choosing $\tau=\rho\wedge\gamma_*$ in \eqref{supermart_gammainfty} for some $\rho\in\CT$, we see that also the process $e^{-r (t \wedge \gamma_*)}V(X_{t \wedge \gamma_*},Y_{t \wedge \gamma_*})$, $t\ge0$ is a super-martingale as claimed. As it is a non-negative super-martingale, Fatou's lemma gives 
\[
V(x,y)\ge \liminf_{t\to\infty} \EE_{x,y}[ e^{-r (t\wedge \gamma_*)}V(X_{t \wedge \gamma_*},Y_{t \wedge \gamma_*})]\ge \EE_{x,y}[ e^{-r \gamma_*}V(X_{\gamma_*},Y_{\gamma_*})]
\]
hence the super-martingale is closed.

Finally, we prove that $\gamma_*$ is optimal for the seller, i.e.~player 2. We have
\begin{align*}
V(x,y) &\ge \EE_{x,y}[ e^{-r (\tau \wedge \gamma_*)}V(X_{\tau \wedge \gamma_*},Y_{\tau \wedge \gamma_*})]\\
&= \EE_{x,y}[ e^{-r \tau }V(X_{\tau},Y_{\tau })\indic_{\{\tau \leq \gamma_*\}}]+\EE_{x,y}[ e^{-r \gamma_* }G_2(X_{\gamma_*})\indic_{\{\tau > \gamma_*\}}]\\
&\ge \EE_{x,y}[ e^{-r \tau }G_1(X_{\tau})\indic_{\{\tau \leq \gamma_*\}}]+\EE_{x,y}[ e^{-r \gamma_* }G_2(X_{\gamma_*})\indic_{\{\tau > \gamma_*\}}]\\
&=M_{x,y}(\tau,\gamma_*)
\end{align*}
Taking the supremum over $\tau$ gives the optimality of strategy $\gamma_*$ for player 2.

It remains to prove the sub-martingale property. Let us denote
\begin{equation} \label{region1n}
\CS_1^{(n)}=\big\{ (x,y) \in \RR_+\times [0,1]\, \tq\, V^{(n)}(x,y)=G_1^{(n)}(x) \big\},
\end{equation}
the stopping region of player $1$. Notice that an analogous set can be defined relatively to $V$ and $G_1$ (see \eqref{region1}). In Section \ref{sec:regions} properties of $\CS_1$ are proven in Lemma \ref{monotonic} by using continuity and monotonicity of $V$. The same methodology can be applied to $V^{(n)}$ to prove analogous properties for $\CS_1^{(n)}$. To be precise it is worth noticing that \eqref{right-c} holds for $V^{(n)}$ provided that $x\le x'\le n$ therein. The rest of (iii) in Lemma \ref{monotonic} follows by recalling that $[n,+\infty)\times[0,1]\subseteq \CS_1$. The analogy holds with Corollary \ref{cor:b} as well. In particular there exists a non-increasing lower-semi-continuous map $b_1^{(n)}: [0,1] \rightarrow \RR_+$ such that for $x\ge K$ it holds $(x,y)\in \CS_1^{(n)} \Leftrightarrow x \geq b_1^{(n)}(y)$.

Observe that if $(x,y) \in \CS^{(n+1)}_1$ is such that $x<n$, we have 
\[V^{(n)}(x,y)\leq V^{(n+1)}(x,y)=G_1^{(n+1)}(x)=G_1(x)=G_1^{(n)}(x),\] 
which implies $(x,y)\in \CS_1^{(n)}$. Together with the fact that $[n,\infty)\times [0,1] \subset \CS_1^{(n)}$, this implies that $\CS_1^{(n+1)}\subset \CS_1^{(n)}$. By the same arguments, we prove that $\CS_1^{(n)} \subset \CS_1$. We deduce that the sequence $b_1^{(n)}$ is non-decreasing and that
$\tau_n$ is a non-decreasing sequence of stopping times such that $\tau_\infty := \lim_n \tau_n \leq \tau_*$. Moreover, if $(x,y) \in \RR_+ \times [0,1)$ is such that $x<b_1(y)$, then $V(x,y)> G_1(x)$ and, for sufficiently large $n$, we have $V^{(n)}(x,y) > G_1(x)=G_1^{(n)}(x)$, implying that $x<b_1^{(n)}(y)$. We deduce that $b_1^{(n)}$ converges to $b_1$ on $[0,1)$ pointwise.

Now, we prove that $\tau_\infty=\tau_*$. Since  $\tau_\infty\leq \tau_*$, it is sufficient to show that the equality holds $\PP_{x,y}$-almost surely on $\{\tau_\infty <\infty\}$. For $(x,y)\in \CS_1$ the claim is trivial. Fix $(x,y)\notin \CS_1$ and $\omega \in \{\tau_\infty <\infty\}$. Since the sequence $(b^{(n)}_1)_{n\in\NN}$ is non-decreasing then for fixed $m\in\NN$ and any $n\ge m$ we have $b_1^{(n)}(Y_{\tau_n}(\omega))\ge b_1^{(m)}(Y_{\tau_n}(\omega))$. The latter implies
\begin{align*}
\liminf_{n\to\infty}b_1^{(n)}(Y_{\tau_n}(\omega))\ge b_1^{(m)}(Y_{\tau_\infty}(\omega))
\end{align*}
by using that $Y_{\tau_n}(\omega)\to Y_{\tau_\infty}(\omega)$ as well.
Taking the supremum over $m$ in the right-hand side of the above expression and recalling that $b^{(m)}_1\uparrow b_1$ pointwise we conclude
\begin{align}\label{liminf}
\liminf_{n\to\infty} b_1^{(n)}(Y_{\tau_n}(\omega))\ge b_1(Y_{\tau_\infty}(\omega)).
\end{align}

Since $X_{\tau_n}\ge b_1(X_{\tau_n})$, $\PP_{x,y}$-a.s.~for all $n\in\NN$, using continuity of paths and \eqref{liminf}, we also find
\begin{align*}
X_{\tau_\infty}=\lim_{n\to\infty}X_{\tau_n}\ge \liminf_{n\to\infty}b^{(n)}_1(Y_{\tau_n})\ge b_1(Y_{\tau_\infty})\quad\text{$\PP_{x,y}$-a.s.}
\end{align*}
which implies $\tau_\infty\geq \tau_*$, $\PP_{x,y}$-a.s.~as requested.

Finally we notice that the process $e^{-r(t\wedge \tau_n)} V^{(n)}(X_{t\wedge \tau_n},Y_{t \wedge \tau_n})$ is a sub-martingale for all $n$. Since $\tau_n\le \tau_{n+p}$ for all $n,p\ge 0$, we deduce that 
\[
V^{(n+p)}(x,y) \leq \EE_{x,y}[e^{-r(t\wedge \tau_n)} V^{(n+p)}(X_{t\wedge \tau_n},Y_{t \wedge \tau_n})].
\]
Letting $p\to\infty$, monotone convergence implies that 
\[ V(x,y) \leq \EE_{x,y}[e^{-r(t\wedge \tau_n)} V(X_{t\wedge \tau_n},Y_{t \wedge \tau_n})]\]
for all $n$. Taking $n\to\infty$ and recalling that $e^{-r(t\wedge \tau_n)} V(X_{t\wedge \tau_n},Y_{t \wedge \tau_n}) \leq \sup_{s \in [0,t]} e^{-rs} X_s \in L^1(\PP_{x,y})$, bounded convergence implies
\[ V(x,y) \leq \EE_{x,y}[e^{-r(t\wedge \tau^*)} V(X_{t\wedge \tau^*},Y_{t \wedge \tau^*})].\]
The above result and the Markov property imply that $e^{-r(t\wedge \tau^*)} V(X_{t\wedge \tau^*},Y_{t \wedge \tau^*})$ is a sub-martingale as claimed. 

\section{Proof of Lemma \ref{lem:hitC}}\label{app:lemhitC}
The proof is more easily carried out considering the boundaries $c_1$ and $c_2$ rather than $b_1$ and $b_2$. However we incur no loss of generality thanks to the equivalence of the problem formulation with respect to the coordinates $(x,y)$ and $(z,y)$. We provide a full argument for $k> 0$ but a completely symmetric proof holds for $k<0$.

Since $c_1$ is non-decreasing and $Y$ is non-degenerate at all points of $(0,1)$, the law of iterated logarithm implies that $\hat\tau_*=\check \tau$, $\mathbb{P}$-a.s. Similarly if $(Z,Y)$ hits the line $y_K(\cdot)$ from above then it will immediately cross it downwards.

For the same result relative to the boundary $c_2$ we repeat the steps in \cite[Cor.~8]{CP15}. In particular let us introduce some notation
\begin{align} 
&\hat{\gamma}_\eps:=\inf\{t>0\,|\,Y_t\ge c_2(Z_t)+\eps\},\quad \hat{\gamma}^\delta_\eps:=\inf\{t>\delta\,|\,Y_t\ge c_2(Z_t)+\eps\}\\[+3pt]
&\check \gamma_\eps:=\inf\{t> 0\,|\,Y_t> c_2(Z_t)+\eps\},\quad \check{\gamma}^\delta_\eps:=\inf\{t>\delta\,|\,Y_t> c_2(Z_t)+\eps\}
\end{align}
so that $\hat \gamma_*=\hat\gamma_0$ and $\check\gamma=\check \gamma_0$. We have $\hat\gamma_+:=\lim_{\eps\to0}\hat\gamma_\eps=\check\gamma$ and 
\[
\hat\gamma^\delta_0\le \hat\gamma^\delta_+:=\lim_{\eps\to 0}\hat\gamma^\delta_\eps=\check \gamma^\delta_0.
\]
Assume that for any $(z,y)\in R_K$ we have
\begin{align}\label{CP1}
\mathbb{P}_{z,y}(\check \gamma^\delta_0>t)\le \mathbb{P}_{z,y}(\hat\gamma^\delta_0>t)
\end{align}
so that $\check\gamma^\delta_0=\hat\gamma^\delta_0$, $\mathbb{P}_{z,y}$-a.s. Then 
\begin{align*}
\check\gamma=\lim_{\eps\to 0}\hat\gamma_\eps=\lim_{\eps\to 0}\lim_{\delta\to 0}\hat\gamma^\delta_\eps=&\lim_{\delta\to 0}\lim_{\eps\to 0}\hat\gamma^\delta_\eps=\lim_{\delta\to 0}\check \gamma^\delta_0
=\lim_{\delta\to 0}\hat \gamma^\delta_0=\hat\gamma_0=\hat\gamma_*
\end{align*}
where the last limit is easily verified by definition of $\hat \gamma^\delta_0$ and we could swap the limits because $\hat \gamma^\delta_\eps$ is non-decreasing in both $\delta$ and $\eps$.
\vspace{+4pt}

Now it remains to verify \eqref{CP1}. We start by noticing that any interval of the form $(\delta,t)$ may be decomposed into the union of countably many intervals over which $c_2$ is either strictly increasing or flat. Consider the latter, i.e.~let $\CI\subset \mathbb{R}$ be an interval such that $c_2(\zeta)=y_0$ for $\zeta\in\CI$ and a fixed $y_0\in(0,1)$. Fix also $(z,y)\in R_K$, then it is immediate to check that on the event $\{\hat\gamma_*\in\CI\}$ one has $\hat\gamma_*=\check\gamma$, $\mathbb P_{z,y}$-a.s., because $Y$ immediately crosses $y_0$ after reaching it. This in particular implies that 
\begin{align}\label{CP2}
\mathbb P_{z,y}\left(Y_s\le c_2(Z_s),\,\forall s\in\CI\right)=\mathbb P_{z,y}\left(Y_s< c_2(Z_s),\,\forall s\in\CI\right).
\end{align}
Next we fix $h_0\in(0,\delta/2)$ so that for $h\in(0,h_0)$ we have $c_2(Z_s)\le c_2(Z_{s+h})$, $\mathbb P_{z,y}$-a.s., because $c_2$ and $Z$ are non-decreasing. Moreover \emph{the inequality is strict} whenever $c_2$ is strictly increasing.  Hence, the latter consideration and \eqref{CP2} imply 
\begin{align*}
\mathbb{P}_{z,y}(\check\gamma^\delta_0>t)=&\mathbb P_{z,y}\left(Y_s\le c_2(Z_s),\,\forall s\in(\delta,t]\right)\\
\le& \mathbb P_{z,y}\left(Y_s< c_2(Z_{s+h}),\,\forall s\in(\delta,t]\right)\\
=&\mathbb P_{z,y}\left(Y_{r-h}< c_2(Z_{r}),\,\forall r\in(\delta+h,t+h]\right)\\
\le&\mathbb P_{z,y}\left(Y_{r-h}< c_2(Z_{r}),\,\forall r\in(\delta+h_0,t]\right)\\
\le&\mathbb P_{y}\left(Y_{r-h}< c_2(z+kr),\,\forall r\in(\delta+h_0,t]\right),
\end{align*}
where in the last expression we have expressed $Z$ explicitly so that it can be treated effectively as a `time' variable. 

We now denote by $p_Y$ and $m_Y$ the probability transition density and the speed measure of $Y$, respectively. Then by using the Markov property of $Y$ we obtain
\begin{align*}
\mathbb{P}_{z,y}(\check\gamma^\delta_0>t)\le&\mathbb P_{y}\left(Y_{r-h}< c_2(z+kr),\,\forall r\in(\delta+h_0,t]\right)\\
=&\mathbb E_{y}\left[\mathbb P_{Y_{\delta/2-h}}\left(Y_{r-\delta/2}<c_2(z+kr),\,\forall r\in(\delta+h_0,t]\right)\right]\\
=&\int_0^1p_Y(\delta/2-h,y,\xi)\mathbb P_{\xi}\left(Y_{r-\delta/2}<c_2(z+kr),\,\forall r\in(\delta+h_0,t]\right)m_Y(d\xi).
\end{align*}
Scheff\'e's theorem (see page 224 in \cite{Bil}) guarantees that 
\begin{align*}
\lim_{h\to 0}\int_0^1|p_Y(\delta/2-h,y,\xi)-p_Y(\delta/2,y,\xi)|m_Y(d\xi)=0
\end{align*}
thus implying that taking limits as $h\to 0$ we obtain
\begin{align*}
\mathbb{P}_{z,y}(\check\gamma^\delta_0>t)\le&\int_0^1p_Y(\delta/2,y,\xi)\mathbb P_{\xi}\left(Y_{r-\delta/2}<c_2(z+kr),\,\forall r\in(\delta+h_0,t]\right)m_Y(d\xi)\\
=&\mathbb P_{y}\left(Y_{r}<c_2(z+kr),\,\forall r\in(\delta+h_0,t]\right)=\mathbb{P}_{z,y}(\hat\gamma^{\delta+h_0}_0>t).
\end{align*}
Letting now $h_0\to0$ we find \eqref{CP1} as claimed, because it is easy to verify that $\hat\gamma^{\delta+h_0}_0\downarrow\hat\gamma^{\delta}_0$.

\section{Game with complete information: summary of results}

In this appendix we provide a short summary of existing results concerning the stopping regions in the game call option problem with perfect information, i.e.~when $y$ is either $0$ or $1$. The material below is based on results contained in \cite{yam}, for $y=1$, and \cite{ekstromvilleneuve}, for $y=0$.

We recall $M_{x,y}(\tau,\gamma)$ as in \eqref{gameformulationone} and emphasise that here $y=\{0,1\}$.
Denote by $V_{\infty}$ the value of the optimal stopping problem for the buyer when there is no possible seller's cancellation (i.e. when $\gamma=+\infty$):
\[
V_{\infty}(x,y):= \sup_{\tau} M_{x,y}(\tau,+\infty)
\]
and by $V_K$ the value of the problem when $\gamma=\gamma_K$, i.e.~the hitting time of $\{K\}$ by $X$:
\[ 
V_{K}(x,y):= \sup_{\tau} M_{x,y}(\tau,\gamma_K). 
\]
We also define the critical dividend levels $\delta_1 < \delta_2$ by
\[ \delta_1:= \inf\{ \delta>0 \tq \lim_{x \downarrow K} \frac{V_K(x,1)- \varepsilon}{x-K} \leq 1\}\quad\text{and}\quad \delta_2:= \inf \{ \delta>0  \tq V_{\infty}(K,1) < \varepsilon \}.\]

For the process $X$ we recall that the fundamental solutions of 
\[ \tfrac{\sigma^2}{2}x^2 f''(x)+(r-\delta)x f'(x) -rf(x)=0,\quad x>0\]
are $\psi(x)=x^{\lambda_1}$ and $\phi(x)=x^{\lambda_2}$,
with $\psi$ increasing (notice that $\lambda_1 >1$) and $\phi$ decreasing, and where $\lambda_2<\lambda_1$ solve
\[ \tfrac{\sigma^2}{2}\lambda^2 + (r-\delta_0- \frac{\sigma^2}{2}) \lambda - r =0.\] 

The next Proposition summarises results of \cite{yam} and \cite[Sec.~5.1]{ekstromvilleneuve}.
\begin{pro} \label{perfectinfo}
The following four cases hold
\begin{itemize}
\item \textbf{Case 1:} If $\varepsilon_0 \geq K$ we have  
\begin{itemize}
\item $\CS_2 \cap \{y=0 \}=\CS_2\cap \{y=1 \} = \emptyset$,
\item $\CS_1\cap \{y=0 \}=\emptyset$ and $\CS_1\cap \{y=1\}=[\frac{\lambda_1}{\lambda_1 - 1} K,+\infty)$.
\end{itemize}

\item \textbf{Case 2:} If $\varepsilon_0 < K$ and $\delta_0 \geq \delta_2$ we have 
\begin{itemize}
\item $\CS_2 \cap \{y=0 \}=[K, +\infty)$ and $\CS_2 \cap \{y=1 \}=\emptyset$.
\item $\CS_1 \cap \{y=0 \}=\emptyset$ and $\CS_1 \cap \{y=1 \}=[\frac{\lambda_1}{\lambda_1 - 1} K,+\infty)$.
\end{itemize}

\item \textbf{Case 3:} If $\varepsilon_0 <K$ and $\delta_1 \leq \delta_0 < \delta_2$ we have
\begin{itemize}
\item $\CS_2\cap \{y=0 \}=[K,+\infty)$ and $\CS_2 \cap \{y=1 \}=\{K\}$.
\item $\CS_1\cap \{y=0 \}=\emptyset$ and  $\CS_1 \cap \{y=1 \}|=[\alpha_0, +\infty)$, 
where $\alpha_0$ is the unique solution of 
\begin{equation*}
\left(\frac{\alpha_0-K}{\alpha_0}\lambda_1 - 1 \right)\alpha_0^{\lambda_1-\lambda_2}-\left(\frac{\alpha_0-K}{\alpha_0}\lambda_2 - 1 \right)K^{\lambda_1-\lambda_2}= \frac{\varepsilon_0}{K}(\lambda_1-\lambda_2) \alpha_0^{\lambda_1-1}K^{1-\lambda_2}.
\end{equation*}
\end{itemize}

\item \textbf{Case 4:} If $\varepsilon_0 <K$ and $0<\delta_0 < \delta_1$ we have
\begin{itemize}
\item $\CS_2\cap \{y=0 \} = [K,+\infty)$ and $\CS_2\cap \{y=1 \}=[K, \beta_1]$\\
\item $\CS_1 \cap \{y=0 \} = \emptyset$ and $\CS_1\cap \{y=1 \}=[\alpha_1,+\infty)$ where $(\alpha_1,\beta_1)$ is the unique solution of the system of equations
\begin{align*}
\left\{ 
\begin{array}{r}  
\left(\frac{\alpha_1-K}{\alpha_1}\lambda_1 - 1 \right)\alpha_1^{\lambda_1-\lambda_2}-\left(\frac{\alpha_1-K}{\alpha_1}\lambda_2 - 1 \right)\beta_1^{\lambda_1-\lambda_2}=\frac{\beta_1-K+\varepsilon_0}{\beta_1}(\lambda_1-\lambda_2) \alpha_1^{\lambda_1-1}\beta_1^{1-\lambda_2}\\[+12pt] 
\left(\frac{\beta_1-K+\varepsilon_0}{\beta_1}\lambda_1 - 1 \right)\beta_1^{\lambda_1-\lambda_2}-\left(\frac{\beta_1-K+\varepsilon_0}{\beta_1}\lambda_2 - 1 \right)\alpha_1^{\lambda_1-\lambda_2}= \frac{\alpha_1-K}{\alpha_1}(\lambda_1-\lambda_2) \beta_1^{\lambda_1-1}\alpha_1^{1-\lambda_2}
\end{array} 
\right.
\end{align*}
\end{itemize}
\end{itemize}
For all the above cases, the pair $(\gamma_*,\tau_*)$ defined in Theorem \ref{value} is not a saddle point if $y=0$, and is a saddle point if $y=1$.
\end{pro}

\end{document}